\documentclass{article}

\usepackage[style=trad-plain,backref=true,url=false,isbn=false,alldates=year,sortcites=true,maxnames=9]{biblatex}
\renewbibmacro{in:}{%
  \ifentrytype{article}{}{\printtext{\bibstring{in}\intitlepunct}}}
\usepackage{geometry}
\usepackage{amssymb}
\usepackage{latexsym, amsmath, amscd,amsthm}
\usepackage{newtxtext,newtxmath}
\usepackage{mathtools}
\usepackage{thmtools}
\usepackage{graphicx}
\usepackage[percent]{overpic}
\usepackage{units}
\usepackage{tikz}
\usepackage{diagbox}
\usepackage{xcolor}
\definecolor{linkblue}{HTML}{003d73}
\definecolor{linkgreen}{HTML}{006161}
\definecolor{linkred}{HTML}{a11950}
\usepackage{hyperref}
\hypersetup{ 
	pdftitle={Direct Sampling of Confined Polygons in Linear Time},
	pdfauthor={Clayton Shonkwiler and Kandin Theis},
	pdfsubject={knot theory},
	pdfkeywords={knots, random polygons, confined polymers},
	colorlinks=true,
	linkcolor=linkblue,
	citecolor=linkgreen,
	urlcolor=linkred
}
\PassOptionsToPackage{caption=false,labelformat=empty}{subfig}
\usepackage[lofdepth]{subfig}
\usepackage[export]{adjustbox}
\usepackage{algorithm}
\usepackage{algorithmicx}
\usepackage{algpseudocode}
\usepackage{booktabs}
\usepackage{authblk}

\usepackage{supertabular,multicol,ifthen,multirow}

\usepackage{array}
\newcolumntype{L}[1]{>{\raggedright\let\newline\\\arraybackslash\hspace{0pt}}m{#1}}
\newcolumntype{C}[1]{>{\centering\let\newline\\\arraybackslash\hspace{0pt}}m{#1}}
\newcolumntype{R}[1]{>{\raggedleft\let\newline\\\arraybackslash\hspace{0pt}}m{#1}}

\usepackage[nameinlink]{cleveref}

\makeatletter
\let\mcnewpage=\newpage
\newcommand{\TrickSupertabularIntoMulticols}{%
\renewcommand\newpage{%
    \if@firstcolumn%
        \hrule width\linewidth height0pt%
            \columnbreak%
        \else%
          \mcnewpage%
        \fi%
}%
}
\makeatother

\graphicspath{{./figs/}}

\newtheorem{theorem}{Theorem}
\newtheorem{lemma}[theorem]{Lemma}
\newtheorem{proposition}[theorem]{Proposition}
\newtheorem{corollary}[theorem]{Corollary}

\theoremstyle{definition}
\newtheorem{definition}[theorem]{Definition}
\newtheorem*{example}{Example}
\newtheorem{conjecture}[theorem]{Conjecture}

\newtheorem*{notation}{Notation}

\newcommand{\R}{\mathbb{R}}

\newcommand{\Pol}{\operatorname{Pol}}
\newcommand{\Polhat}{\widehat{\operatorname{Pol}}}
\newcommand{\SO}{\operatorname{SO}}
\newcommand{\vol}{\operatorname{vol}}
\newcommand{\one}{\mathbf{1}}

\newcommand{\seqnum}[1]{\href{http://oeis.org/#1}{#1}}
\newcommand{\ent}[2]{E_{#1,#2}}
\newcommand{\gent}[3]{E_{#1,#2}^{(#3)}}

\let\oldReturn\Return
\renewcommand{\Return}{\State\oldReturn}

\DefineBibliographyStrings{english}{%
  backrefpage = {$\uparrow$},
  backrefpages = {$\uparrow$},
  page = {p\adddot},
  pages = {pp\adddot},
}

\DeclareSourcemap{
  \maps{
    \map{
      \step[fieldset=pagetotal, null]
	  \step[fieldset=pubstate, null]
    }
  }
}

\DeclareFieldFormat{eprint:urn}{%
  \mkbibacro{URN}\addcolon\space
  \ifhyperref
    {\href{https://nbn-resolving.org/urn:#1}{\nolinkurl{#1}}}
    {\nolinkurl{#1}}}
\DeclareFieldAlias{eprint:URN}{eprint:urn}
\DeclareFieldFormat{eprint:hal}{%
  \mkbibacro{HAL}\addcolon\space
  \ifhyperref
    {\href{https://hal.science/#1}{\nolinkurl{#1}}}
    {\nolinkurl{#1}}}
\DeclareFieldAlias{eprint:HAL}{eprint:hal}
\DeclareFieldFormat{eprint:numdam}{%
  Numdam\addcolon\space
  \ifhyperref
    {\href{http://www.numdam.org/item/#1}{\nolinkurl{#1}}}
    {\nolinkurl{#1}}}
\DeclareFieldAlias{eprint:Numdam}{eprint:numdam}
\DeclareFieldFormat{eprint:ark}{%
  \mkbibacro{ARK}\addcolon\space
  \ifhyperref
    {\href{https://n2t.net/ark:#1}{\nolinkurl{#1}}}
    {\nolinkurl{#1}}}
\DeclareFieldAlias{eprint:ARK}{eprint:ark}
\DeclareFieldFormat{eprint:zbl}{%
  Zbl\addcolon\space
  \ifhyperref
    {\href{https://zbmath.org/#1}{\nolinkurl{#1}}}
    {\nolinkurl{#1}}}
\DeclareFieldAlias{eprint:Zbl}{eprint:zbl}
\DeclareFieldFormat{eprint:mr}{%
  \mkbibacro{MR}\addcolon\space
  \ifhyperref
    {\href{https://mathscinet.ams.org/mathscinet-getitem?mr=#1}{\nolinkurl{#1}}}
    {\nolinkurl{#1}}}
\DeclareFieldAlias{eprint:MR}{eprint:mr}

\tikzset{my node/.style = {shape=circle, fill=black, inner sep = 1.5pt, outer sep = 0pt}}

\title{Direct Sampling of Confined Polygons in Linear Time}
\author{Clayton Shonkwiler}
\author{Kandin Theis}
\affil{Department of Mathematics, Colorado State University, Fort Collins, CO, USA}

\date{}

\begin{document}

\maketitle

\begin{abstract}
	We present an algorithm for sampling tightly confined random equilateral closed polygons in three-space  which has runtime linear in the number of edges. Using symplectic geometry, sampling such polygons reduces to sampling a moment polytope, and in our confinement model this polytope turns out to be natural from a combinatorial point of view. This connection to combinatorics yields both our fast sampling algorithm and explicit formulas for the expected distances of vertices to the origin. We use our algorithm to investigate the expected total curvature of confined polygons, leading to a very precise conjecture for the asymptotics of total curvature.
\end{abstract}

\section{Introduction} 
\label{sec:introduction}

In this paper we introduce the Confined Polygons from Order Polytopes (CPOP) algorithm for sampling random equilateral polygons in tight confinement. This is an interesting mathematical story, using symplectic geometry to reduce the problem to sampling a particular polytope, which turns out to be essentially the order polytope of the zig-zag poset, but it is also relevant to polymer science and biochemistry, as random polygons provide a simple model of ring polymers and the confinement model we use was introduced to try to understand DNA packing in, e.g., viral capsids.

For our purposes, a \emph{polygon} is a piecewise-linear embedding of the circle into $\R^3$, and an \emph{equilateral polygon} is a polygon in which all the segments have unit length. We can think of a random polygon as a random walk which is conditioned to close up and form a loop. 

In the polymer literature, random walks in $\R^3$ are often called \emph{freely-jointed chains} and provide a simple, well-understood, yet somewhat robust model for polymers in solution~\cite{rayleighProblemRandomVibrations1919,floryStatisticalMechanicsChain1989,hughesRandomWalksRandom1995}. In turn, random polygons provide the analogous model for ring polymers~\cite{orlandiniStatisticalTopologyClosed2007}, including many biopolymers like most bacterial DNA. 

However, biopolymers are often confined into small volumes and it remains an active area of research to understand how, for example, viral DNA gets packed into viral capsids~\cite{landerDNABendinginducedPhase2013,marenduzzoDNADNAInteractions2009,michelettiSimulationsKnottingConfined2008,arsuagaDNAKnotsReveal2005,arsuagaKnottingProbabilityDNA2002,jardineDNAPackagingDoublestranded2005,reithEffectiveStiffeningDNA2012}. This motivates the question of how to sample random equilateral polygons in tight confinement.

The closure condition imposes correlations between the edge directions which makes sampling polygons rather challenging: algorithms for sampling (unconfined) random equilateral polygons have been proposed for at least four decades~\cite{alvaradoGenerationRandomEquilateral2011,kleninEffectExcludedVolume1988,plunkettTotalCurvatureTotal2007,varelaFastErgodicAlgorithm2009,millettKnottingRegularPolygons1994,mooreTopologicallyDrivenSwelling2004,mooreLimitsAnalogySelfavoidance2005,vologodskiiStatisticalMechanicsSupercoils1979,cantarellaSymplecticGeometryClosed2016,cantarellaFastDirectSampling2016,cantarellaFasterDirectSampling2024,cantarellaCoBarSFastReweighted2024,diaoGeneratingEquilateralRandom2011,diaoGeneratingEquilateralRandom2012a,diaoGeneratingEquilateralRandom2012}. Of particular note are the Progressive Action-Angle Method (PAAM)~\cite{cantarellaFasterDirectSampling2024}, which produces random equilateral $n$-gons in $\Theta(n^2)$ time, and the Conformal Barycenter Sampler (CoBarS)~\cite{cantarellaCoBarSFastReweighted2024}, which produces reweighted samples in $\Theta(n)$ time.

Random confined polygons must satisfy not only the closure constraint which says that the path must eventually get back to its starting point, but also the confinement constraints, which have a variable effect depending on locations of vertices. The simplest confinement constraint is \emph{rooted spherical confinement of radius} $R$, in which the first vertex is placed at the center of a sphere of radius $R$, and all subsequent vertices are required to lie within this sphere. For equilateral polygons with unit-length edges, $R \geq 1$ since the second and last vertices must lie at distance 1 from the first vertex. 

In this regime there are two notable algorithms in use: Diao, Ernst, Montemayor, and Ziegler's method~\cite{diaoGeneratingEquilateralRandom2012a} based on generating successive marginals of the joint distribution of vertex distances from the center of the confining sphere, and the Toric Symplectic Markov Chain Monte Carlo (TSMCMC) algorithm introduced in~\cite{cantarellaSymplecticGeometryClosed2016}. Both have disadvantages: Diao et al.'s algorithm is rather computationally expensive, so it is challenging to generate large ensembles of confined $n$-gons when $n$ is large; on the other hand, TSMCMC can generate large ensembles, but it is a Markov chain with essentially useless bounds on its rate of convergence, so the effective size of an ensemble is hard to estimate.

CPOP combines the best of both worlds in the case $R=1$: it directly samples uniformly distributed and independent confined $n$-gons like Diao et al., so there is no concern about convergence rates, and it is even faster than TSMCMC. Indeed, as we will see in \Cref{thm:sampling}, CPOP generates random equilateral $n$-gons in rooted spherical confinement of radius 1 in time $\Theta(n)$, which is theoretically optimal, and in practice this is quite fast: for example, we can generate confined 20,000-gons at a rate of about 500/second on a desktop computer (see \Cref{fig:sample timings}). The drawback relative to other methods is that (so far at least) CPOP only applies in the case $R=1$, which is the smallest possible radius for rooted spherical confinement. 

Here is an overview of the structure of the paper: in \Cref{sec:polygons and polytopes} we use the symplectic geometry of polygon space~\cite{kapovichSymplecticGeometryPolygons1996} to reduce the problem of sampling equilateral $n$-gons in rooted confinement of radius~1 to the problem of sampling a particular polytope $\mathcal{P}_n(1)$. Up to an affine transformation, this polytope is equivalent to the order polytope of the zig-zag poset~\cite{stanleyTwoPosetPolytopes1986}, which is triangulated by simplices indexed by alternating permutations. This combinatorial connection is developed in \Cref{sec:polytopes and permutations}. In \Cref{sec:sampling}, we exploit the combinatorics to define CPOP using an algorithm of Marchal~\cite{marchalGeneratingRandomAlternating2012} for sampling $\mathcal{P}_n(1)$, which was originally part of an algorithm for quickly sampling alternating permutations. In addition to sampling, this approach yields a precise characterization (\Cref{cor:asymptotic chordlength}) of the asymptotic distributions of vertex distances from the origin in the large-$n$ limit.

In \Cref{sec:combinatorics} we take a bit of a combinatorial digression, defining \emph{augmented zig-zag posets} and giving three formulas for the number of linear extensions of these posets, one in terms of Entringer numbers, one recursive, and the last in terms of \emph{generalized Entringer numbers}, which we also define. Along the way, we prove a new identity for the Entringer numbers (\Cref{cor:entringer sum}) and determine the expected value of $\tau(i)$, where $\tau$ is a random alternating permutation (see \Cref{cor:expected first entry,cor:expected ith entry,cor:expected ith entry pt 2}); these results may be of some independent interest. \Cref{sec:chord lengths} connects this back to polygons, giving formulas for the expected distance to the origin of the $i$th vertex of a random confined $n$-gon in terms of this combinatorial data. We use CPOP to generate large ensembles of confined $n$-gons in \Cref{sec:numerics}, with the goal of understanding the expected total curvature of these polygons. We find strong evidence that the expected total curvature is asymptotic to $\left(\frac{\pi}{2} + 0.57545\right)n - 0.46742$ (see \Cref{conj:turning angle}), which is compatible with but rather more precise than Diao, Ernst, Rawdon, and Ziegler's model for expected total curvature of confined polygons~\cite{diaoTotalCurvatureTotal2018}. Finally, \Cref{sec:conclusion} concludes with some open questions and possible avenues for future investigation.


\section{Polygons and Polytopes}
\label{sec:polygons and polytopes}

As mentioned in the introduction, a \emph{polygon} is a piecewise-linear mapping of the circle into $\R^3$, and an \emph{equilateral polygon} is such an embedding for which each linear segment has the same length. Up to scaling, we may as well assume that length is 1. We can represent a polygon by the locations of its vertices $v_1, \dots , v_n \in \R^3$, and the equilateral polygon condition is equivalent to requiring that $|v_{i+1} - v_i| = 1$ for $i=1,\dots , n-1$ and $|v_1 - v_n| = 1$.

We can instead think in terms of edge vectors $e_1, \dots , e_n$, where $e_i = v_{i+1} - v_i$ for $i=1, \dots , n-1$ and $e_n = v_1 - v_n$. The edges of an equilateral polygon will be unit vectors, so we can equivalently think of an equilateral polygon as a collection $(e_1, \dots , e_n) \in (S^2)^n$ of unit vectors. The fact that these edges fit together to form a closed loop means that they satisfy the vector equation $e_1 + \dots + e_n = 0$, so the collection of equilateral polygons
\[
	\Pol(n) := \{(e_1, \dots , e_n) \in (S^2)^n: e_1 + \dots + e_n = 0\}
\]
forms a codimension-3 subset of the product $(S^2)^n \subset (\R^3)^n$ of $n$-tuples of unit vectors. One can show that the subset $\Pol(n)^\times$ on which all edges are distinct forms a smooth $(2n-3)$-dimensional submanifold of $(S^2)^n$ and that the $(2n-3)$-dimensional Hausdorff measure of $\Pol(n) \backslash \Pol(n)^\times$ vanishes; in this sense, $\Pol(n)$ is almost everywhere a submanifold of $(S^2)^n$ which inherits the submanifold Riemannian metric and associated volume measure.

Thinking in terms of edge vectors gives a translation-invariant representation of polygons; if we also want a rotation-invariant representation, it is natural to take the quotient
	\[
		\Polhat(n) := \Pol(n)/\SO(3),
	\]
which we call \emph{equilateral polygon space}. In turn, this quotient has a Riemannian metric defined by the condition that $\Pol(n) \to \Polhat(n)$ is a Riemannian submersion, and hence a natural probability measure given by normalizing the Riemannian volume.

We now describe some coordinates on $\Polhat(n)$ with respect to which this probability measure has a particularly simple form. These coordinates are inspired by symplectic geometry~\cite{kapovichSymplecticGeometryPolygons1996}, but are entirely elementary.

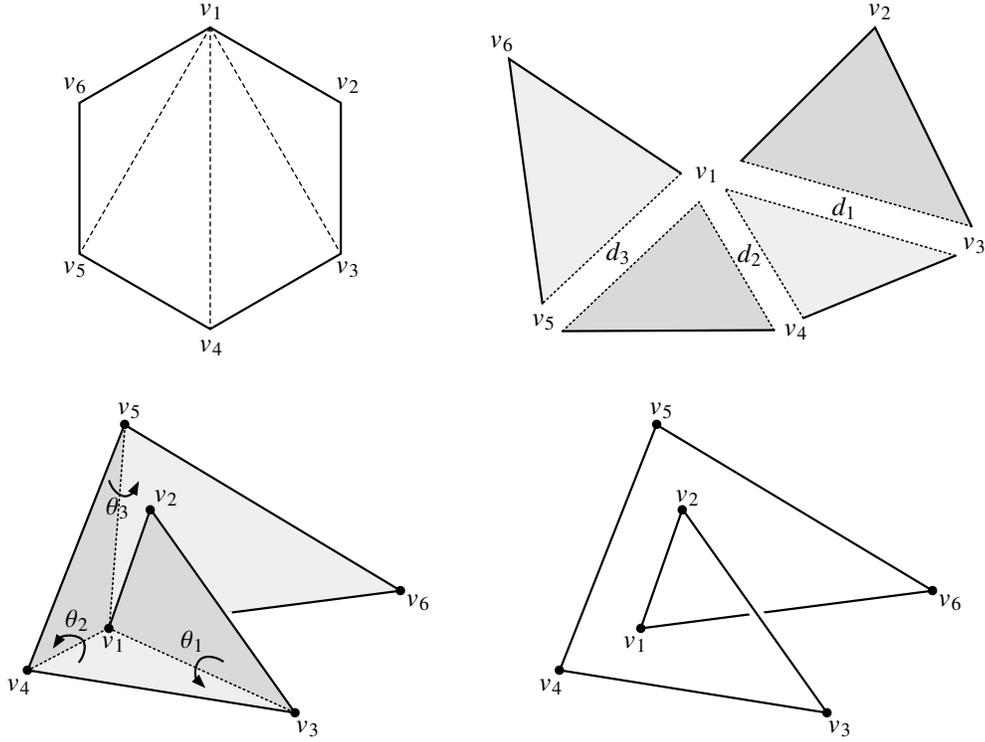
\begin{figure}[t]
	\centering
		\begingroup
		\setlength{\unitlength}{2.75in}
	    \begin{picture}(1,0.75)
			\put(0,0){\includegraphics[width=\unitlength]{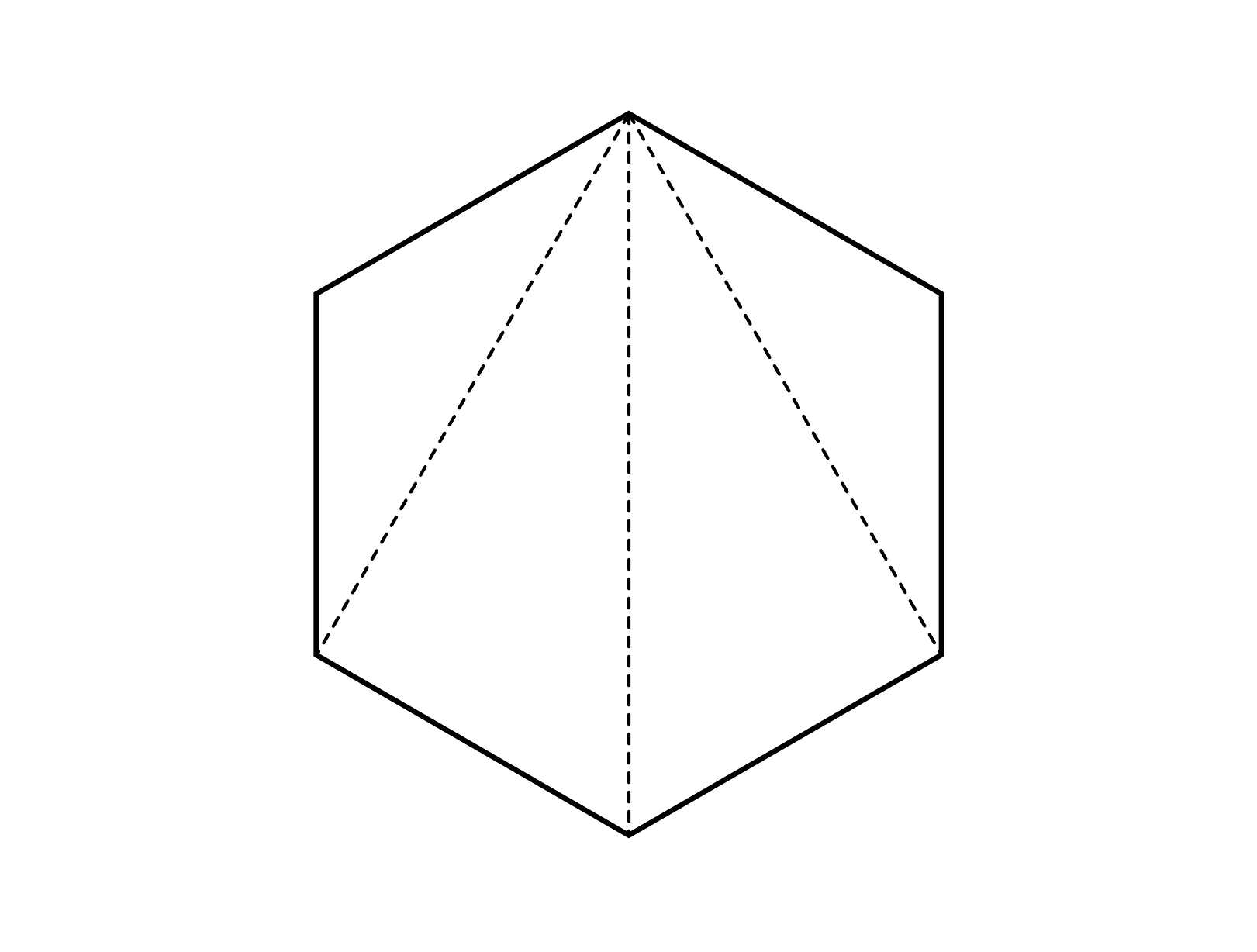}}
		  	\put(0.48,0.69){\smash{$v_1$}}
			\put(0.74,0.55){\smash{$v_2$}}
			\put(0.74,0.2){\smash{$v_3$}}
			\put(0.48,0.05){\smash{$v_4$}}
			\put(0.22,0.2){\smash{$v_5$}}
			\put(0.22,0.55){\smash{$v_6$}}
	    \end{picture}
	    \begin{picture}(1,0.75)
	      \put(0,0){\includegraphics[width=\unitlength]{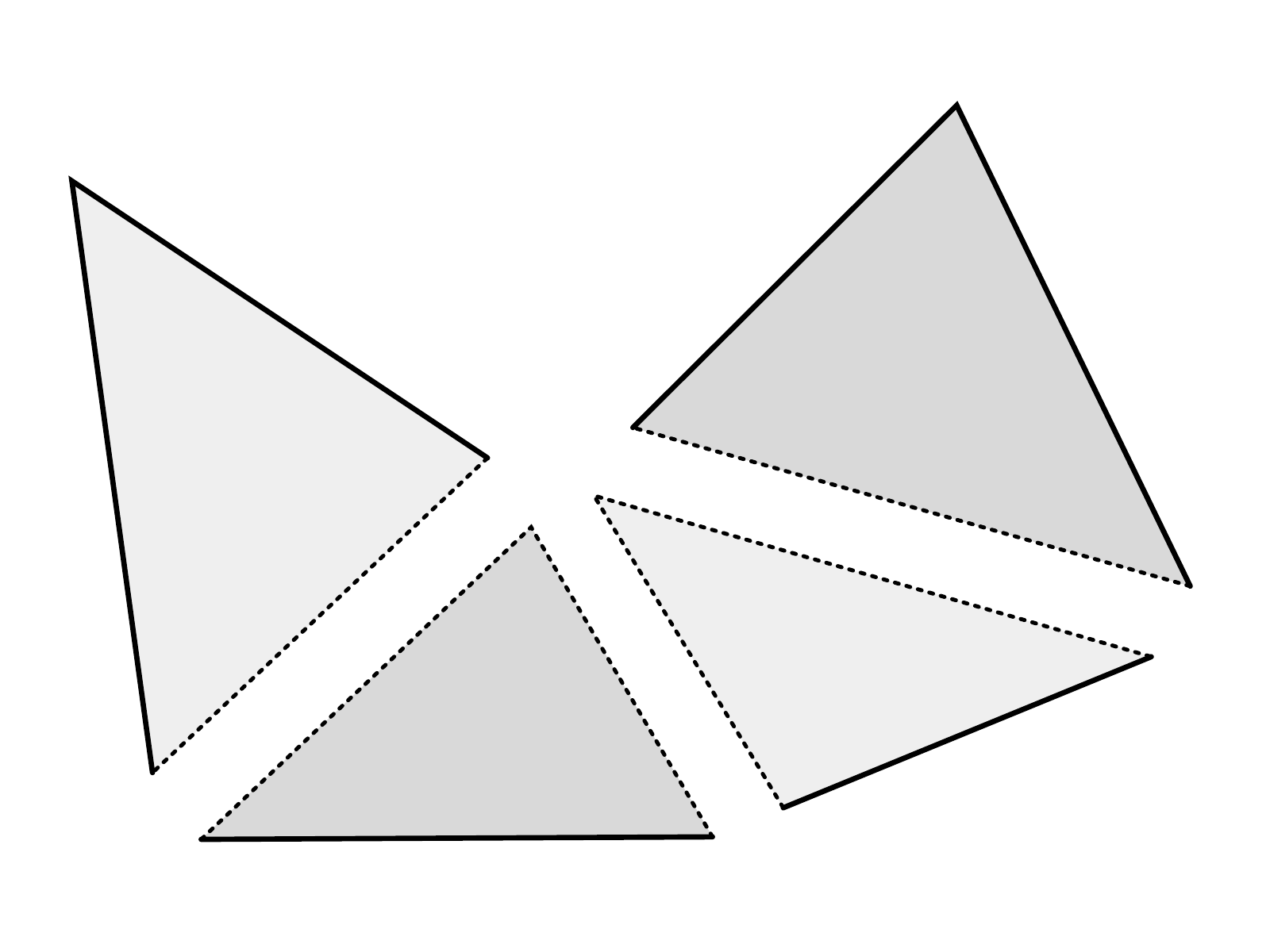}}
	      \put(0.41,0.38){\smash{$v_1$}}
	      \put(0.74,0.69){\smash{$v_2$}}
	      \put(0.92,0.245){\smash{$v_3$}}
	      \put(0.58,0.08){\smash{$v_4$}}
	      \put(0.1,0.1){\smash{$v_5$}}
	      \put(0.02,0.63){\smash{$v_6$}}
	      \put(0.67,0.31){\smash{$d_1$}}
	      \put(0.49,0.22){\smash{$d_2$}}
	      \put(0.24,0.225){\smash{$d_3$}}
	    \end{picture}
	    \begin{picture}(1,0.75)
	      \put(0,0){\includegraphics[width=\unitlength]{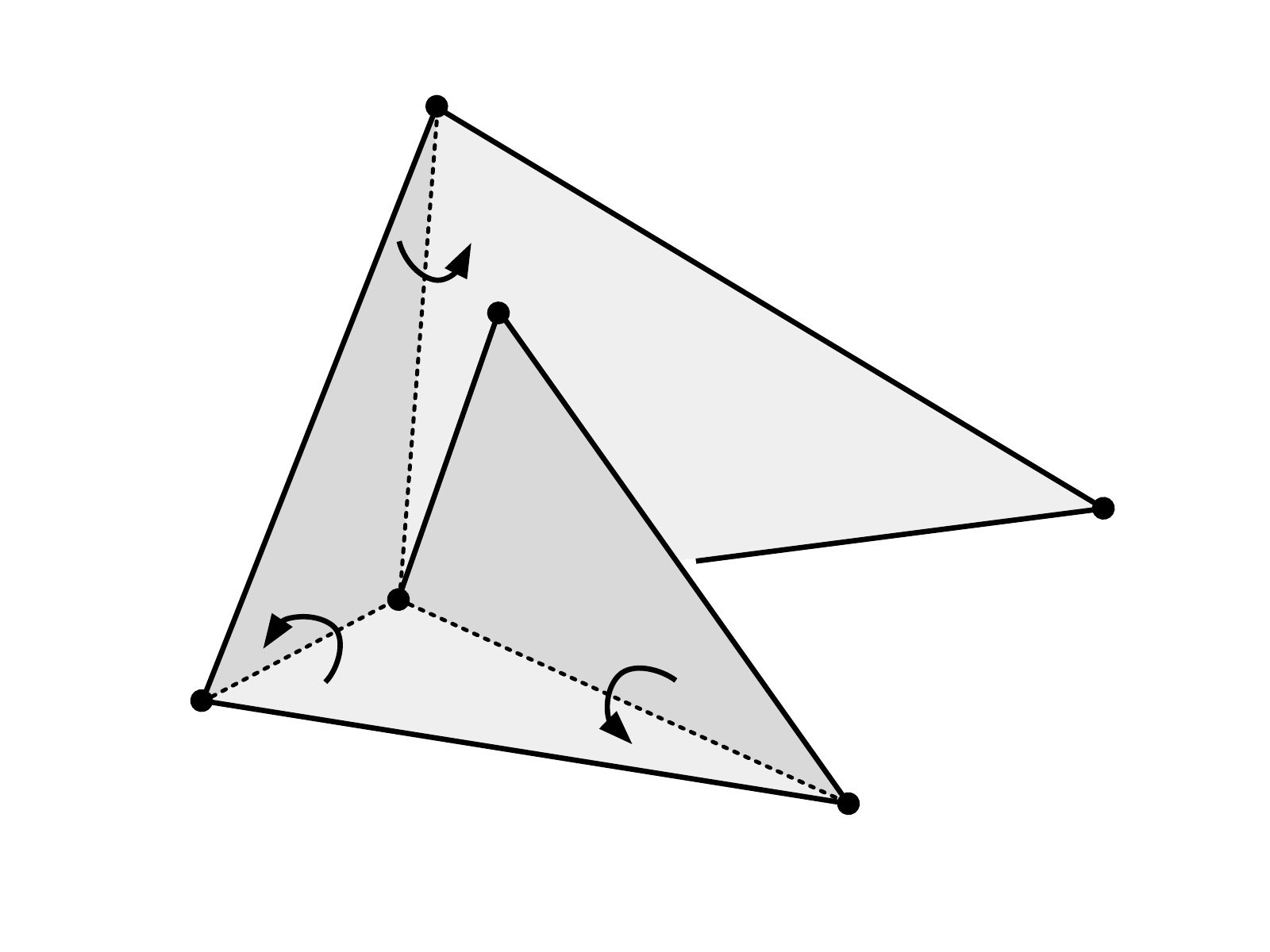}}
	      \put(0.3,0.24){\smash{$v_1$}}
	      \put(0.4,0.52){\smash{$v_2$}}
	      \put(0.67,0.08){\smash{$v_3$}}
	      \put(0.12,0.16){\smash{$v_4$}}
	      \put(0.33,0.685){\smash{$v_5$}}
	      \put(0.88,0.32){\smash{$v_6$}}
	      \put(0.45,0.24){\smash{$\theta_1$}}
	      \put(0.23,0.28){\smash{$\theta_2$}}
	      \put(0.308,0.495){\smash{$\theta_3$}}
	    \end{picture}
	    \begin{picture}(1,0.75)
	      \put(0,0){\includegraphics[width=\unitlength]{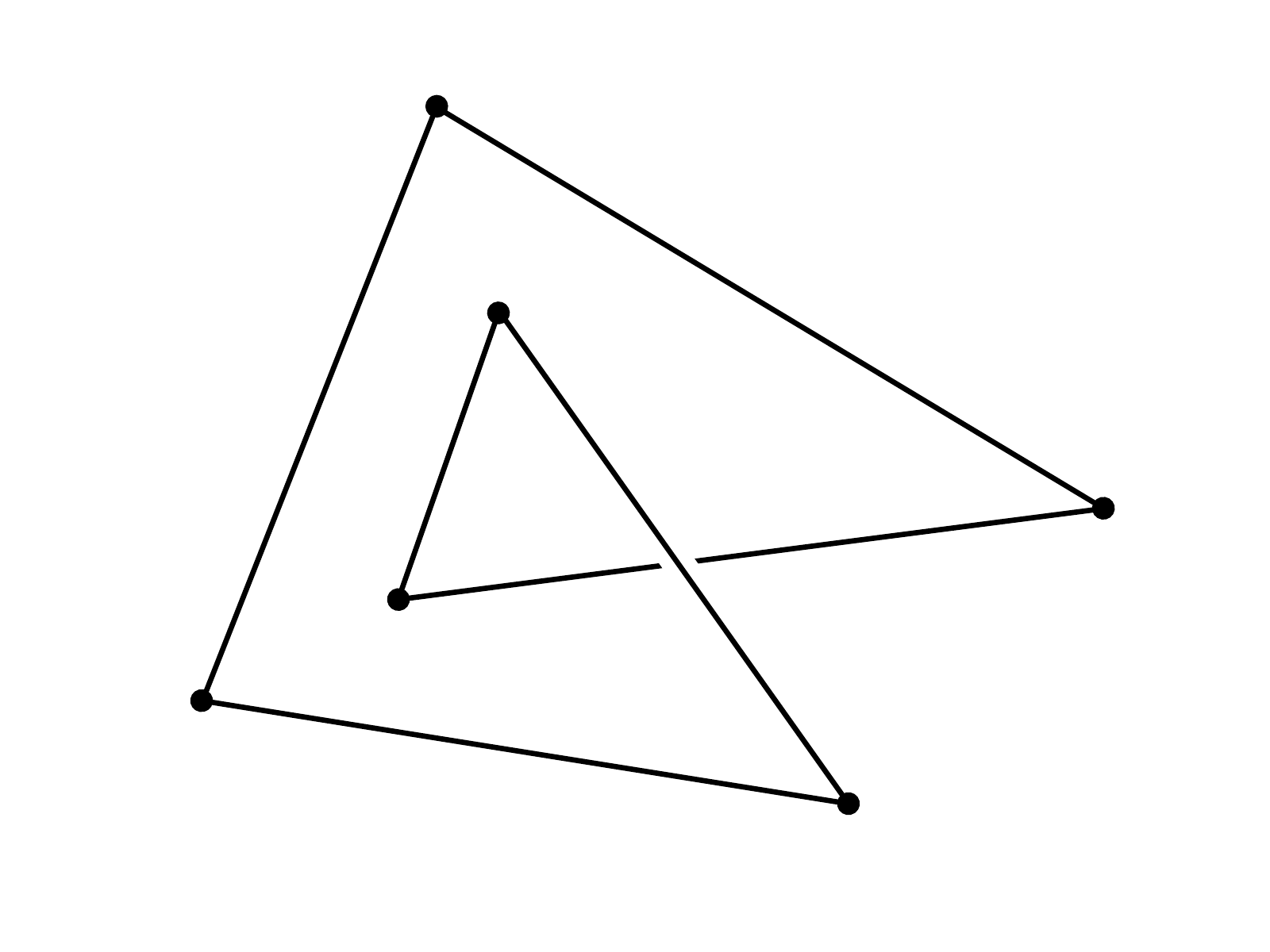}}
	      \put(0.28,0.24){\smash{$v_1$}}
	      \put(0.38,0.52){\smash{$v_2$}}
	      \put(0.67,0.08){\smash{$v_3$}}
	      \put(0.12,0.16){\smash{$v_4$}}
	      \put(0.33,0.685){\smash{$v_5$}}
	      \put(0.88,0.32){\smash{$v_6$}}
	    \end{picture}
		\endgroup
	\caption{Illustrating the reconstruction map $\alpha: \mathcal{P}_n \times T^{n-3}$ which takes diagonal lengths and dihedral angles to an equilateral polygon. Top left shows the triangulation of an abstract hexagon. Given $d_1, d_2, d_3$ which obey the triangle inequalities~\eqref{eq:fan polytope}, build the four triangles in the triangulation from their side lengths (top right). Given dihedral angles $\theta_1, \theta_2, \theta_3$, we can build a piecewise-linear surface out of these triangles (bottom left). The boundary of this surface is the resulting polygon in space (bottom right).}
	\label{fig:fan triangulation}
\end{figure}

Connecting $v_3, \dots , v_{n-1}$ to $v_1$---as in \Cref{fig:fan triangulation} (top left)---decomposes an abstract polygon into $n-2$ triangles. Those triangles are determined up to congruence by their side lengths, which are given by $|e_1| = \dots = |e_n| = 1$ and by the diagonal lengths $d_i := |v_{i+2} - v_1|$ for $i=1, \dots , n-3$. Hence, if we view an equilateral polygon as the boundary of the piecewise linear surface whose faces are these triangles---as in \Cref{fig:fan triangulation} (bottom left)---then the geometry of the surface (and hence also of its boundary) is completely determined by $d_1, \dots , d_{n-3}$ and the dihedral angles $\theta_1, \dots , \theta_{n-3}$ between adjacent triangles. Therefore, $d_1, \dots , d_{n-3}, \theta_1, \dots , \theta_{n-3}$ give a system of (almost-everywhere defined) coordinates on $\Polhat(n)$.

In symplectic terms, these coordinates are \emph{action-angle coordinates} on $\Polhat(n)$, as first explained by Kapovich and Millson~\cite{kapovichSymplecticGeometryPolygons1996}. While the dihedral angles can be chosen completely independently, the diagonal lengths cannot be independent, as they must satisfy the following system of triangle inequalities:
\begin{equation}
0 \leq d_1 \leq 2
\qquad
\begin{matrix}
1 \leq d_i + d_{i+1} \\
-1 \leq d_{i+1} - d_i \leq 1
\end{matrix}
\qquad
0 \leq d_{n-3} \leq 2.
\label{eq:fan polytope}
\end{equation}

Let $\mathcal{P}_n \subset \R^{n-3}$ be the convex polytope determined by the above inequalities. Letting $T^{n-3} = (S^1)^{n-3}$ be the product of unit circles, the action-angle coordinates $((d_1, \dots , d_{n-3}),(\theta_1, \dots , \theta_{n-3}))$ are defined on $\mathcal{P}_n \times T^{n-3}$, and the product of Lebesgue measure on $\mathcal{P}_n$ and the standard product measure on $T^{n-3}$ is measure-theoretically equivalent to the natural measure on $\Polhat(n)$:

 \begin{theorem}[{Cantarella--Shonkwiler~\cite{cantarellaSymplecticGeometryClosed2016}}]\label{thm:measure equivalence}
	The map $\alpha: \mathcal{P}_n \times T^{n-3} \to \Polhat(n)$ defining the action-angle coordinates (as illustrated in \Cref{fig:fan triangulation}) is measure-preserving.
\end{theorem}

In consequence, to sample equilateral $n$-gons according to the natural measure on $\Polhat(n)$, it suffices to sample $(d_1, \dots , d_{n-3})$ from Lebesgue measure on $\mathcal{P}_n$ and $(\theta_1,\dots , \theta_{n-3})$ from the product measure on $T^{n-3}$. The latter is straightforward, so the only challenge is sampling from Lebesgue measure on $\mathcal{P}_n$. In \cite{cantarellaFasterDirectSampling2024,cantarellaFastDirectSampling2016} we showed how to do this efficiently, yielding an algorithm for sampling equilateral $n$-gons in expected time $\Theta(n^2)$.\footnote{In practice, the fastest algorithm for Monte Carlo integration on $\Polhat(n)$ is Cantarella and Schumacher's CoBarS algorithm~\cite{cantarellaCoBarSFastReweighted2024,cantarellaComputingConformalBarycenter2022}, which uses reweighted sampling.}

Of course, this is only an algorithm for sampling unconfined polygons. We now add a confinement restriction, and consider equilateral polygons in \emph{rooted spherical confinement}, meaning that all vertices lie within a sphere of radius $R$ centered at a chosen (root) vertex. Given our choice of coordinates, it is most natural to let $v_1$ be the root vertex, so that the confinement condition is simply $d_i \leq R$ for $i=1, \dots , n-3$. Adding these inequalities to the triangle inequalities~\eqref{eq:fan polytope} defines a new polytope $\mathcal{P}_n(R)$. Letting $\Polhat(n;R)$ be the space of equilateral $n$-gons in rooted spherical confinement of radius $R$, we have the following analog of \Cref{thm:measure equivalence}:

\begin{theorem}[{Cantarella--Shonkwiler~\cite{cantarellaSymplecticGeometryClosed2016}}]\label{thm:confined measure equivalence}
	The action-angle map $\alpha: \mathcal{P}_n(R) \times T^{n-3} \to \Polhat(n;R)$ is measure-preserving.
\end{theorem}

We showed in~\cite{cantarellaSymplecticGeometryClosed2016} how to use the hit-and-run Markov chain~\cite{bonehConstraintsRedundancyFeasible1979,smithEfficientMonteCarlo1984,andersenHitRunUnifying2007} on $\mathcal{P}_n(R)$ to get an ergodic Markov chain on $\Polhat(n;R)$, and subsequently used this Markov chain to generate large ensembles of tightly-confined polygons which provided new bounds on stick numbers and superbridge indices of a large number of knots~\cite{eddyImprovedStickNumber2019,eddyNewStickNumber2022,blairKnotsExactly102020,shonkwilerNewComputationsSuperbridge2020,shonkwilerAllPrimeKnots2022,shonkwilerNewSuperbridgeIndex2022a}. 

In contrast, Diao, Ernst, Montemayor, and Ziegler~\cite{diaoGeneratingEquilateralRandom2011} developed a non-Markov chain method for sampling $\Polhat(n;R)$ by generating successive marginals of Lebesgue measure on $\mathcal{P}_n(R)$ and used it to perform various experiments on this model~\cite{diaoGeneratingEquilateralRandom2012a,diaoGeneratingEquilateralRandom2012,diaoRandomWalksPolygons2014,diaoKnotSpectrumConfined2014,diaoTotalCurvatureTotal2018,diaoRelativeFrequenciesAlternating2018,ernstKnottingSpectrumPolygonal2021,diaoAverageCrossingNumber2018}.

The above approaches are generic, in the sense that they work for any $R \geq 1$ (of course, the radius $R$ of the confining sphere cannot be less than 1 since $v_2$ is always at distance 1 from $v_1$), though as pointed out in the introduction, both have substantial drawbacks. In this paper, our goal is to show that we get a dramatically better sampling algorithm in the tightest possible confinement, namely when $R=1$.

\section{Polytopes and Alternating Permutations}
\label{sec:polytopes and permutations}

To do so, note that the confinement inequalities $d_i \leq 1$ for all $i=1, \dots , n-3$ make most of the inequalities in~\eqref{eq:fan polytope} redundant: we can simplify the defining inequalities of $\mathcal{P}_n(1)$ as follows:
\begin{equation}\label{eq:confined inequalities}
	0 \leq d_i \leq 1 \qquad 1 \leq d_i + d_{i+1}.
\end{equation}
In other words
\[
	\mathcal{P}_n(1) = \{(d_1, \dots , d_{n-3}) \in [0,1]^{n-3} : d_i + d_{i+1} \geq 1 \text{ for all } i = 1, \dots , n-4\}.
\]

Though presented in lightly-disguised form, this is a well-known polytope in combinatorics. Under the affine transformation $\varphi: (d_1, \dots , d_{n-3}) \mapsto (d_1, 1-d_2, d_3, 1-d_4, \dots)$, $\mathcal{P}_n(1)$ maps to the polytope
\[
	\mathcal{O}_{n-3} := \{(x_1, \dots , x_{n-3}) \in [0,1]^{n-3} : x_1 \geq x_2 \leq x_3 \geq x_4 \leq \dots \},
\]
which Stanley~\cite[Section~3.8]{stanleySurveyAlternatingPermutations2010} calls the \emph{zig-zag order polytope}, since it is the order polytope of the zig-zag (or fence) poset~\cite[Example~4.3]{stanleyTwoPosetPolytopes1986}.

A map $\sigma: \mathcal{O}_{n-3} \to \{1,\dots, n-3\}$ is order-preserving if and only if the permutation $i \mapsto \sigma(x_i)$ is \emph{alternating}, which we define as follows:

\begin{definition}\label{def:alternating permutation}
	A permutation $\tau \in S_m$ is \emph{alternating} (or \emph{down-up}) if $\tau(1) > \tau(2) < \tau(3) > \dots$, and \emph{reverse-alternating} (or \emph{up-down}) of $\tau(1) < \tau(2) > \tau(3) < \dots$. Let $AP_m \subset S_m$ be the set of alternating permutations, and let $RP_m \subset S_m$ be the set of reverse-alternating permutations.
\end{definition}

Any permutation $\tau \in S_{n-3}$ determines a simplex
\begin{equation}\label{eq:orthoscheme definition}
	\Delta_\tau = \{(x_1, \dots , x_{n-3}) \in [0,1]^{n-3} : x_{\tau^{-1}(1)} \leq x_{\tau^{-1}(2)} \leq \dots \leq x_{\tau^{-1}(n-3)}\}.
\end{equation}
Each such $\Delta_\tau$ is congruent to the standard orthoscheme $\{(x_1, \dots , x_{n-3}) \in [0,1]^{n-3} : x_1 \leq x_2 \leq \dots \leq x_{n-3}\}$, which has volume $\frac{1}{(n-3)!}$. Moreover, when $\tau$ is alternating we have that $\Delta_\tau \subset \mathcal{O}_{n-3}$. Since the $\Delta_\tau$ have disjoint interiors and union equal to $\mathcal{O}_{n-3}$, we see that $\vol(\mathcal{O}_{n-3}) = \frac{E_{n-3}}{(n-3)!}$, where $E_{n-3}$ is the number of alternating (or reverse-alternating) permutations in $S_{n-3}$ (see also~\cite[Exercise~3.66(c)]{stanleyEnumerativeCombinatorics2011}). In turn, since the affine transformation $\varphi:\mathcal{P}_n(1)\to \mathcal{O}_{n-3}$ has determinant $\pm 1$ and hence is volume-preserving, we see that $\vol(\mathcal{P}_n(1)) = \frac{E_{n-3}}{(n-3)!}$ as well.

The number $E_{n-3}$ is known as an \emph{Euler number} (or \emph{up-down number}, \emph{zig-zag number}, or, depending on the parity of $n$, \emph{secant} or \emph{tangent number}), and appears as \seqnum{A000111} in the On-Line Encyclopedia of Integer Sequences (OEIS)~\cite{oeis} (see also \seqnum{A000364} and \seqnum{A000182} for the even and odd terms, respectively). The exponential generating function for the Euler numbers has been known since the 1870s, when André~\cite{andreDeveloppements$operatornamesacuteecX$1879,andrePermutationsAlternees1881} proved that
\[
	\sum_{n=0}^\infty E_n \frac{x^n}{n!} = \sec x + \tan x.
\]

See Stanley's excellent survey~\cite{stanleySurveyAlternatingPermutations2010} for a wealth of information about alternating permutations and Euler numbers. For example, the Euler numbers satisfy the simple recurrence
\[
	2E_{n+1} = \sum_{k=0}^n \binom{n}{k}E_k E_{n-k}
\]
and have the convergent asymptotic series
\begin{equation}\label{eq:euler asymptotics}
	\frac{E_n}{n!} = 2\left(\frac{2}{\pi}\right)^{n+1} \sum_{k=0}^\infty (-1)^{k(n+1)} \frac{1}{(2k+1)^{n+1}}.
\end{equation}
In particular, this implies that $\frac{E_n}{n!}$ is asymptotically equivalent to $2\left(\frac{2}{\pi}\right)^{n+1}$ (cf.~\cite[Example~IV.35]{flajoletAnalyticCombinatorics2009}) and hence that
\[
	\vol(\mathcal{P}_n(1)) = \frac{E_{n-3}}{(n-3)!} \sim 2 \left(\frac{2}{\pi}\right)^{n-2}.
\]
In other words, $\mathcal{P}_n(1)$ is an exponentially small subset of the unit hypercube $[0,1]^{n-3}$, so the naïve sampling strategy of rejection sampling the hypercube is not computationally feasible for large $n$.

Alternatively, the transformation $\psi: (d_1, \dots , d_{n-3}) \mapsto (1-d_1, \dots , 1-d_{n-3})$ maps $\mathcal{P}_n(1)$ to
\[
	\mathcal{C}_{n-3} := \{(x_1, \dots , x_{n-3}) \in [0,1]^{n-3} : x_i + x_{i+1} \leq 1 \text{ for all } i=1, \dots , n-4\},
\]
the \emph{zig-zag chain polytope} (again, see \cite[Section~3.8]{stanleySurveyAlternatingPermutations2010} or~\cite[Example~4.3]{stanleyTwoPosetPolytopes1986}). The polytope $\mathcal{C}_{n-3}$ seems to have been first considered in questions posed by Stanley~\cite{stanleyElementaryProblemE27011978} and Doberkat~\cite{doberkatHypervolumeProblem84201984}.

By way of the maps $\varphi$ and $\psi$, the polytopes $\mathcal{P}_n(1)$, $\mathcal{O}_{n-3}$, and $\mathcal{C}_{n-3}$ all have the same volume and are combinatorially equivalent; this is a special case of a theorem of Stanley~\cite[Theorem~2.3]{stanleyTwoPosetPolytopes1986}. In particular, they all have the same number of vertices, namely $F_{n-1}$, the $(n-1)$st Fibonacci number~\cite[Exercise~1.35(e)]{stanleyEnumerativeCombinatorics2011}. More generally, the full $f$-vector of these polytopes is known~\cite[Corollary~3.5]{chebikin$f$vectorDescentPolytope2011}.

\section{Sampling}
\label{sec:sampling}

Since the $\Delta_\tau$ which decompose $\mathcal{O}_{n-3}$ are indexed by alternating permutations and since sampling from Lebesgue measure on simplices is straightforward, any fast algorithm for sampling random alternating permutations can be adapted to give an algorithm for sampling from Lebesgue measure on $\mathcal{P}_n(1)$. For example, there are fast Boltzmann samplers~\cite{bodiniBoltzmannSamplersFirstorder2012,duchonBoltzmannSamplersRandom2004} for generating uniformly random alternating permutations on $\approx n$ letters, and a quadratic time algorithm~\cite{marchalDensityMethodPermutations2018} for generating uniformly random permutations with prescribed descent set based on the density method~\cite{banderierRectangularYoungTableaux2018,marchalRectangularYoungTableaux2020}.

The fastest algorithm for directly sampling alternating permutations seems to be that of Bodini, Durand, and Marchal~\cite{bodiniOptimalGenerationStrictly2024}, which is a modification of an earlier algorithm of Marchal~\cite{marchalGeneratingRandomAlternating2012}. Both of these algorithms produce random alternating permutations in $O(n \log n)$ time, but in examining the details, it turns out that both are actually sampling points from $\mathcal{P}_n(1)$ in $O(n)$ time, applying $\varphi$ to get points in $\mathcal{O}_{n-3}$, and then sorting the coordinates to get alternating permutations. Indeed, Marchal's older algorithm~\cite{marchalGeneratingRandomAlternating2012} is most appropriate for our purposes, since it generates samples from Lebesgue measure on $\mathcal{P}_n(1)$.

The basic observation underlying Marchal's approach is that a sequence of random draws $u_1, u_2, \dots$ from Lebesgue measure on $[0,1]$ can be interpreted as a Markov chain, where the transition probability is just uniform on $[0,1]$ (independent of the previous state). The initial segment $(u_1, \dots , u_{n-3})$ is likely to fail one of the defining inequalities $u_i + u_{i+1} \geq 1$, so it is unlikely to be in $\mathcal{P}_n(1)$; if we want to get points in $\mathcal{P}_n(1)$ we should forbid certain transitions. The simplest way to do this is, given the present state $x \in [0,1]$, to define a transition measure $P_x: \mathcal{B} \to [0,x]$ on the Borel sets of $[0,1]$ by
\[
	P_x(A) := \lambda(A \cap [1-x,1])
\]
where $\lambda$ is Lebesgue measure. This is not a probability measure unless $x=1$; to get a probability measure, we might hope to find a conditional density $f_x$ supported on $[1-x,1]$ so that
\[
	Q_x(A) := \int_{A} f_x(y)d\lambda(y)
\]
is a probability measure, and hence a Markov kernel, and we could generate a sequence $x_1, x_2, \dots$ by letting $x_1 = u_1 \sim U([0,1])$ and sampling $x_{i+1}$ from $Q_{x_i}$ for each $i=1,2,\dots$.

The resulting joint density of the initial sequence $(x_1, \dots , x_{n-3})$ is then
\begin{equation}\label{eq:joint density product}
	f_{x_{n-4}}(x_{n-3}) \dots f_{x_1}(x_2) \one_{[0,1]},
\end{equation}
which means that the obvious choice of just letting $f_x(y) = \one_{\{y \geq 1-x\}} \frac{1}{x}$ (that is, normalized Lebesgue measure on $[1-x,1]$) is not promising.

Instead, the expression~\eqref{eq:joint density product} for the joint density suggests that a conditional density $f_x$ of the form
\begin{equation}\label{eq:conditional density form}
	f_x(y) = \one_{\{y \geq 1-x\}} a \frac{h(y)}{h(x)}
\end{equation}
for some function $h$ might work well, since then the product in~\eqref{eq:joint density product} telescopes:
\begin{equation}\label{eq:telescope1}
	f_{x_{n-4}}(x_{n-3}) \dots f_{x_1}(x_2) \one_{[0,1]} = \one_{\{x_2 \geq 1-x_1, \dots, x_{n-3} \geq 1-x_{n-4}\}} a^{n-4} \frac{h(x_{n-3})}{h(x_1)},
\end{equation}
which only depends on the first and last states.

Suppose we have a conditional density of the form~\eqref{eq:conditional density form}. While this might seem like a strong constraint on $f_x$, densities of this form are in some sense entropy-maximizing~\cite{bassetVolumetryTimedLanguages2013,bassetMaximalEntropyStochastic2015}, so it is not unreasonable to assume they exist.

While the joint distribution of $(x_1, \dots , x_{n-3})$ is not yet uniform on $\mathcal{P}_n(1)$, we can get a new joint distribution which interchanges the roles of the first and last states by reversing the sequence: the density of the reversed initial sequence $(x_{n-3}, \dots , x_1)$ is
\[
	f_{x_{2}}(x_1) \dots f_{x_{n-3}}(x_{n-4}) \one_{[0,1]} = \one_{\{x_2 \geq 1-x_1, \dots, x_{n-3} \geq 1-x_{n-4}\}} a^{n-4} \frac{h(x_1)}{h(x_{n-3})}.
\]
Then we expect a mixture of the initial sequence and the reversed initial sequence to cancel the dependence on first and last states.

More precisely, let $\alpha := \frac{h(x_{n-3})}{h(x_1)}$ and define the tuple $(d_1, \dots , d_{n-3})$ as follows: 
\begin{itemize}
	\item with probability $\frac{1}{\alpha + \alpha^{-1}}$, let $(d_1, \dots , d_{n-3}) := (x_1, \dots , x_{n-3})$; 
	\item with probability $\frac{1}{\alpha + \alpha^{-1}}$, let $(d_1, \dots , d_{n-3}) := (x_{n-3}, \dots , x_1)$;
	\item and with the complementary probability $1-\frac{2}{\alpha + \alpha^{-1}}$, regenerate $(x_1, \dots , x_{n-3})$ and try again. 
\end{itemize}

Note that $\frac{1}{\alpha + \alpha^{-1}} < \frac{1}{2}$ for any $\alpha > 0$ and, as we will see in \Cref{thm:polytope sampling}, the rejection probability $1-\frac{2}{\alpha + \alpha^{-1}}$ is bounded away from 1, so this procedure is well-defined. When $(d_1, \dots, d_{n-3})$ is chosen in this way, its probability density function is a constant multiple of
\[
	\one_{\{d_2 \geq 1-d_1,  \dots, d_{n-3} \geq 1-d_{n-4}\}} a^{n-4} \left[ \frac{\alpha}{\alpha + \alpha^{-1}} + \frac{\alpha^{-1}}{\alpha + \alpha^{-1}} \right] = \one_{\{d_2 \geq 1-d_1,  \dots, d_{n-3} \geq 1-d_{n-4}\}} a^{n-4} ;
\]
that is, it is just (normalized) Lebesgue measure on $\{d_2 \geq 1-d_1,  \dots, d_{n-3} \geq 1-d_{n-4}\} = \mathcal{P}_n(1)$.

So if we could find a conditional density in the form~\eqref{eq:conditional density form}, this would give an efficient sampling algorithm provided the rejection probability is not too high. Of course, if $f_x$ is a density, then it must integrate to 1, so we have the constraint
\[
	1 = \int_0^1 f_x(y) dy = \int_0^1 \one_{\{y \geq 1-x\}} a \frac{h(y)}{h(x)} dy = \int_{1-x}^1 a \frac{h(y)}{h(x)} dy = \frac{a}{h(x)} \int_{1-x}^1 h(y) dy.
\]
In other words, the function $h$ should solve the following Volterra equation of the second kind:
\[
	h(x) = a \int_{1-x}^1 h(y) dy.
\]
This equation has a unique solution~\cite{brunnerCollocationMethodsVolterra2004} (up to the choice of scale factor $a$) and it is easy to verify that
\[
	h(x) = \sin \left( \frac{\pi}{2} x\right)
\]
solves the equation with $a=\frac{\pi}{2}$, so it must be the unique solution.\footnote{To derive this solution, rather than simply verifying that it is a solution, extend $h$ to a periodic function on $[-2,2]$ and expand it in terms of its Fourier series.} So the conditional probability density function should be
\begin{equation}\label{eq:conditional density}
	f_x(y) = \one_{\{y \geq 1-x\}}\frac{\pi}{2} \frac{\sin\left(\frac{\pi}{2}y \right)}{\sin\left(\frac{\pi}{2} x\right)}.
\end{equation}

Integrating yields the conditional cumulative distribution function
\[
	F_x(y) = \one_{\{y \geq 1-x\}}\left( 1- \frac{\cos\left(\frac{\pi}{2}y \right)}{\sin\left(\frac{\pi}{2} x\right)}\right).
\]
But now it is clear how to sample from this distribution: if $U \sim U[0,1]$, then $F_x^{-1}(U)$ will have cumulative distribution function $F_x$. After some simplification using the fact that sine and cosine differ by a phase shift, 
\[
	F_x^{-1}(U) = 1-\frac{2}{\pi}\arcsin\left((1-U) \sin \left(\frac{\pi}{2}x\right)\right).
\]
If $U\sim U[0,1]$, then so is $1-U$, and it follows that the random variable 
\[
	1-\frac{2}{\pi}\arcsin\left(U\sin \left(\frac{\pi}{2}x\right)\right)
\]
has cumulative distribution function $F_x$, as desired.

This then motivates Marchal's algorithm for sampling from $\mathcal{P}_n(1)$, which we state as \Cref{alg:polytope sampling}.

\begin{algorithm}[H]
\begin{algorithmic}
\Procedure{ConfinedPolytopeSample}{$n$}\Comment{Generate point in $\mathcal{P}_n(1)$}
\Repeat
\State $d_1 \gets $\Call{UniformRandom}{$[0,1]$}
	\For{$i=1$ to $n-4$}
		\State $u_{i+1} \gets $\Call{UniformRandom}{$[0,1]$}
		\State $d_{i+1} \gets 1 - \frac{2}{\pi} \arcsin\left(u_{i+1} \sin\left(\frac{\pi}{2}d_i\right)\right)$
	\EndFor
	\State $\alpha \gets \frac{\sin\left(\frac{\pi}{2} d_{n-3}\right)}{\sin\left(\frac{\pi}{2}d_1\right)}$
	\State $t \gets $\Call{UniformRandom}{$[0,1]$}
	\If{$t < \frac{1}{\alpha + \alpha^{-1}}$}
		\State $(d_1,\dots , d_{n-3}) \gets (d_{n-3},\dots , d_1)$
	\EndIf
\Until{$t < \frac{2}{\alpha + \alpha^{-1}}$}
\Return{$(d_1,\dots , d_{n-3})$}
\EndProcedure
\end{algorithmic}
\caption{$\mathcal{P}_n(1)$ Sampling}
\label{alg:polytope sampling}
\end{algorithm}

The preceding discussion, which is basically just a recapitulation of Marchal's paper, justifies the correctness of this algorithm:

\begin{theorem}[{Marchal~\cite{marchalGeneratingRandomAlternating2012}}]\label{thm:polytope sampling}
	\Cref{alg:polytope sampling} generates random points in $\mathcal{P}_n(1)$ according to Lebesgue measure. The rejection probability $1-\frac{2}{\alpha + \alpha^{-1}}$ is bounded above by $1-\frac{2}{3\pi} \approx 0.787$, so the average complexity is $\Theta(n)$.
\end{theorem}
We do not prove the bound on the rejection probability here; see Marchal's paper for details. Marchal also showed that the rejection probability converges to $1-\frac{8}{\pi^2} \approx 0.1894$ as $n$ gets large, and in practice we see that this convergence is very rapid: see \Cref{fig:rejection probabilities} (left).

\begin{figure}[h]
	\centering
		\includegraphics[height=1.75in]{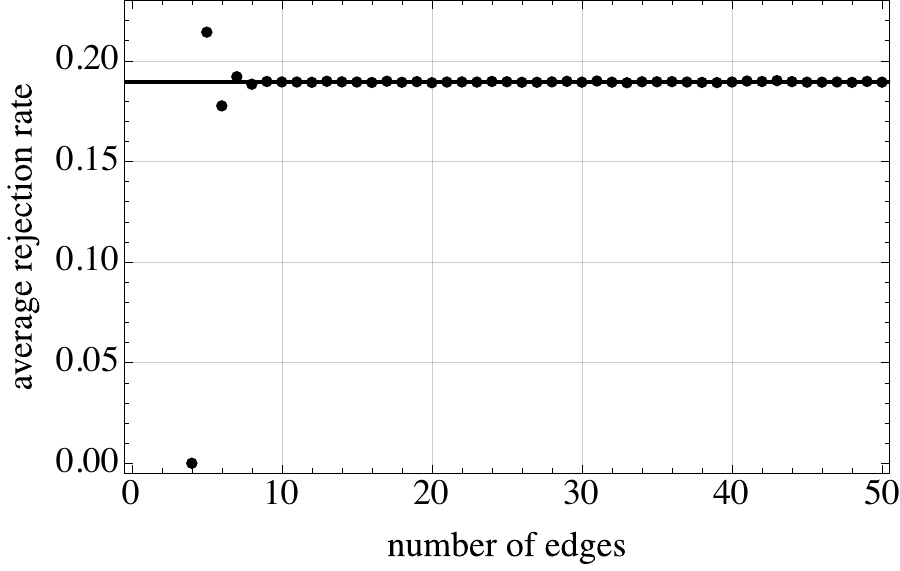}
		\qquad
		\includegraphics[height=1.75in]{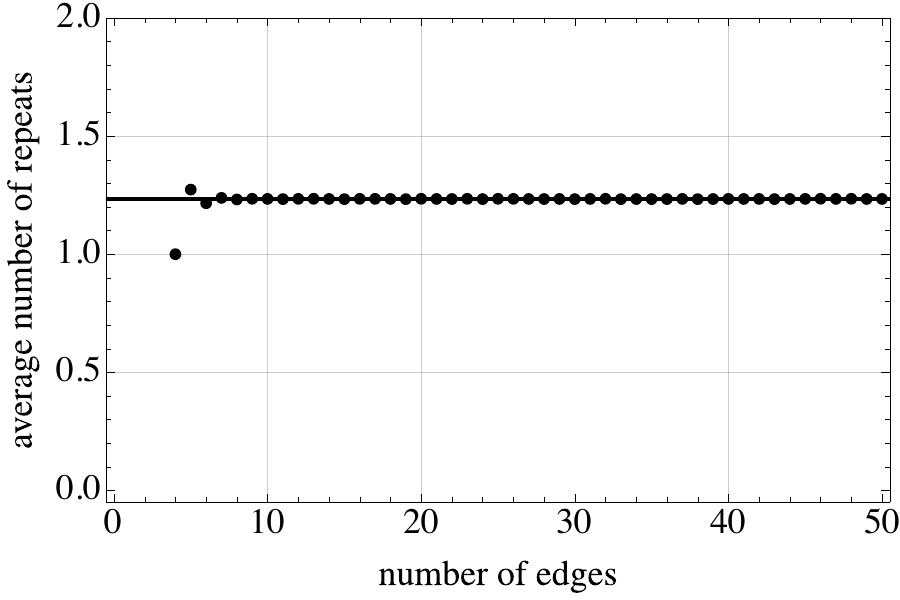}
	\caption{Left: Average rejection probabilities when generating 1,000,000 random points in $\mathcal{P}_n(1)$ with $n=4,\dots , 50$ (dots), compared to the asymptotic limit $1-\frac{8}{\pi^2} \approx 0.1894$ (black line). Right: Average number of repeats of the main loop in \Cref{alg:polytope sampling} when generating 1,000,000 random points in $\mathcal{P}_n(1)$ for the same range of $n$ (dots), compared to the estimate $\frac{\pi^2}{8} \approx 1.2337$ (black line).}
	\label{fig:rejection probabilities}
\end{figure}

Therefore, approximating the rejection probability by the asymptotic value $1-\frac{8}{\pi^2}$, the expected number of times we have to repeat the main loop in \Cref{alg:polytope sampling} is
\[
	\frac{1}{\pi^2/8} \sum_{i=1}^\infty i \left(1-\frac{8}{\pi^2}\right)^{i-1} = \frac{\pi^2}{8} \approx 1.2337,
\]
which matches what we see in practice very closely: see \Cref{fig:rejection probabilities} (right).

Combining \Cref{alg:polytope sampling} with \Cref{thm:confined measure equivalence} using the (linear-time) reconstruction procedure illustrated in \Cref{fig:fan triangulation} yields \Cref{alg:sampling} for sampling from $\Polhat(n;1)$.

\begin{algorithm}[H]
\begin{algorithmic}
\Procedure{ConfinedSample}{$n$}\Comment{Generate confined equilateral $n$-gon}
\State $(d_1,\dots , d_{n-3}) \gets $\Call{ConfinedPolytopeSample}{$n$}
	\For{$i=1$ to $n-3$}
		\State $\theta_i \gets $\Call{UniformRandom}{$[0,2\pi)$}
	\EndFor
	\State Reconstruct polygon $P$ from diagonals $d_1, \dots, d_{n-3}$ and dihedrals $\theta_1, \dots, \theta_{n-3}$.
	\Return{$P$}
\EndProcedure
\end{algorithmic}
\caption{Confined Polygons from Order Polytopes (CPOP) Sampling}
\label{alg:sampling}
\end{algorithm}

\begin{theorem}\label{thm:sampling}
	\Cref{alg:sampling} generates random polygons in $\Polhat(n;1)$ with average time complexity $\Theta(n)$.
\end{theorem}

\Cref{fig:sample timings} shows the time needed to generate 1,000,000 random $n$-gons on an Apple M1 Ultra personal computer using 16 parallel CPU threads with a reference implementation of \Cref{alg:sampling} written as a compiled \emph{Mathematica} function. This implementation is available on Github~\cite{CPOP} and is also included in the supplementary information for this paper.

\begin{figure}[t]
	\centering
		\includegraphics[height=2.5in]{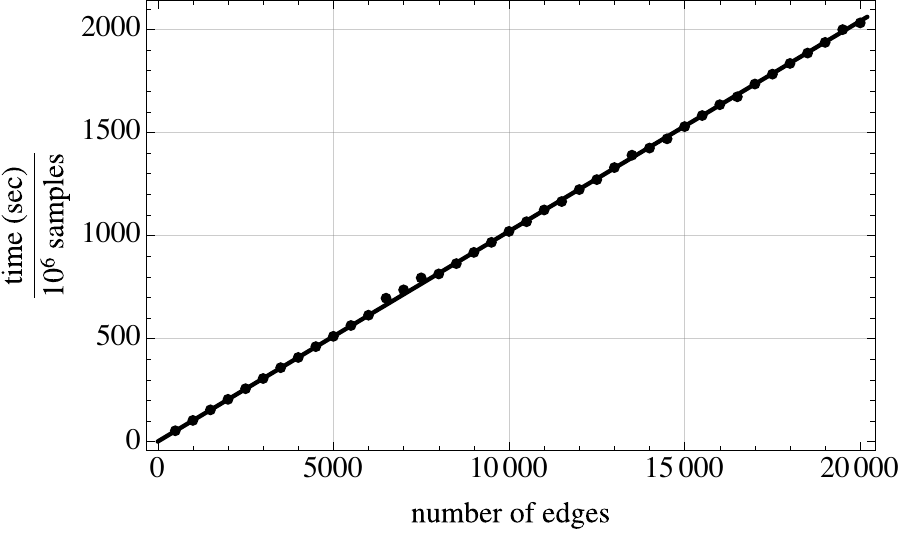}
	\caption{Time per million samples for $n$-gons from $n=500$ to $n=20,\!000$ in steps of $500$. The fitted line has slope $0.102$ (with $R^2 > 0.9999$).}
	\label{fig:sample timings}
\end{figure}

\subsection{Equilibrium Distribution of Chord Lengths}
\label{sec:equilibrium}

As Marchal observes, the invariant measure of the Markov chain $x_1, x_2, \dots$ generated by the conditional density~\eqref{eq:conditional density} is given by
\[
	\mu(A) = \int_A \left(1-\cos\left(\pi y\right)\right)dy
\]
for any Borel set $A$.

To see this, recall that $1-\cos\left(\pi y\right) = 2\sin^2\left(\frac{\pi}{2}y\right)$ and compute for any Borel set $A$
\begin{multline*}
	\int_{[0,1]} P_x(A) \mu(dx) = \int_0^1 \left[\int_A \one_{\{y \geq 1-x\}}\frac{\pi}{2} \frac{\sin\left(\frac{\pi}{2}y \right)}{\sin\left(\frac{\pi}{2} x\right)} dy\right] 2\sin^2\left(\frac{\pi}{2}x\right) dx  \\
	= \int_A \left[ \int_{1-y}^1 \frac{\pi}{2}\sin\left(\frac{\pi}{2} x\right)dx\right]2\sin\left(\frac{\pi}{2} y\right)dy = \int_A 2\sin^2\left(\frac{\pi}{2} y\right)dy = \mu(A).
\end{multline*}

When $n$ is large, this means the $x_i$ generated by \Cref{alg:polytope sampling} (and hence also the $d_i$) will be distributed according to $\mu$.

\begin{corollary}\label{cor:asymptotic chordlength}
When $n$ is large the density of the $d_i$ generated by \Cref{alg:polytope sampling} is asymptotic to
\[
	f(t) = 1-\cos(\pi t).
\]
In particular, when $n$ is large the average of the chord lengths $(d_1, \dots , d_{n-3})$ is asymptotic to
\[
	\frac{1}{2} +\frac{2}{\pi^2} \approx 0.7026.
\]
\end{corollary}

We see these asymptotics in practice, even when $n$ is not particularly large: see~\Cref{fig:large n chordlengths,fig:chordlength asymptotics}.

\begin{figure}[t]
	\centering
		\includegraphics[width=.4\textwidth]{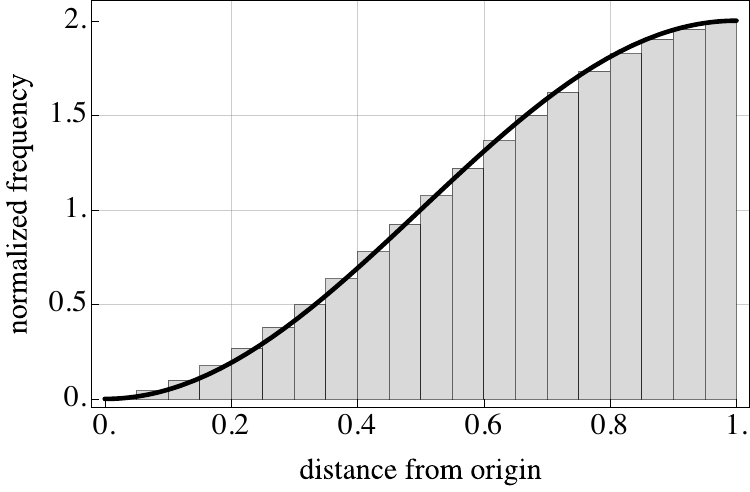}
		\qquad
		\includegraphics[width=.4\textwidth]{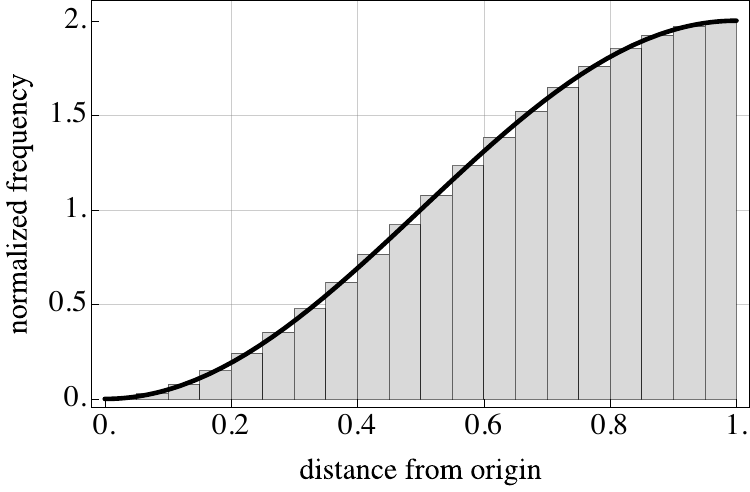}
	\caption{For (left) $n=20$ and $N=1$ million and (right) $n=20,\!000$ and $N=10,\!000$, we generated $N$ samples of $(d_1,\dots, d_{n-3})$ with \Cref{alg:polytope sampling}. The plots show the histograms of the resulting $N(n-3)$ numbers (i.e., all the $d_i$ for all the samples, treated as a single pool of numbers), in both cases plotted against the asymptotic density $f(t) = 1-\cos(\pi t)$.}
	\label{fig:large n chordlengths}
\end{figure}

\section{A Combinatorial Digression}
\label{sec:combinatorics}

\subsection{Entringer Numbers} 
\label{sub:entringer numbers}

We now take a digression into some combinatorial results which will help us to determine expected chord lengths in random confined polygons. First, recall the definition of an Entringer number~\cite{entringerCombinatorialInterpretationEuler1966,seidelEinfacheEntstehungsweiseBernoullischen1877,kempnerShapePolynomialCurves1933,poupardNouvellesSignificationsEnumeratives1982,arnoldBernoulliEulerUpdownNumbers1991}:

\begin{definition}\label{def: entringer}
	Let $0 \leq k \leq n$ and define the \emph{Entringer number} $\ent{n}{k}$ as the number of (down-up) alternating permutations $\tau$ of $\{1, \dots, n+1\}$ with $\tau(1) = k+1$.
\end{definition}

This indexing is chosen, in part, so that $E_{n,n} = E_n$, where we recall that the Euler number $E_n$ counts the number of alternating (or reverse-alternating) permutations of $\{1,\dots, n\}$. Unfortunately, this means that $E_{n,k}$ and $E_n$ are counts of permutations on different base sets, but these notational conventions seem quite standard in the literature and we do not propose to change them here. 

See Conway and Guy~\cite[p.~110]{conwayBookNumbers1996}, Millar, Sloane, and Young~\cite{millarNewOperationSequences1996}, Henry and Wanner~\cite{henryZigzagsBurgiBernoulli2019}, Stanley~\cite[Section~2]{stanleySurveyAlternatingPermutations2010}, and \seqnum{A008280}/\seqnum{A008281}/\seqnum{A008282} in OEIS~\cite{oeis} for more on Entringer numbers, which satisfy the following recurrence (see, e.g.,~\cite[Exercise~1.141]{stanleyEnumerativeCombinatorics2011}):
\begin{equation}\label{eq:Entringer recurrence}
	\ent{0}{0} = 1, \quad \ent{n}{0} = 0 \text{ for } n \geq 1, \quad \ent{n+1}{k+1} = \ent{n+1}{k} + \ent{n}{n-k}\text{ for } n \geq k \geq 0.
\end{equation}

In fact, Entringer~\cite{entringerCombinatorialInterpretationEuler1966} showed that the Entringer numbers have the following formula in terms of Euler numbers:
\begin{equation}\label{eq:entringer formula}
	\ent{n}{k} = \sum_{r=0}^{\lfloor (k-1)/2 \rfloor}(-1)^r \binom{k}{2r+1} E_{n-2r-1}.
\end{equation}

An exponential generating function for the classical Entringer numbers is also known:

\begin{proposition}[{see Graham, Knuth, and Patashnik~\cite[Exercise~6.75]{grahamConcreteMathematicsFoundation1994} and Stanley~\cite[(2.2)]{stanleySurveyAlternatingPermutations2010}}]\label{prop:entringer egf}
	\[
		\sum_{m,n = 0}^\infty E_{m+n,[m,n]} \frac{x^m}{m!}\frac{y^n}{n!}= \frac{\cos x + \sin x}{\cos(x+y)},
	\]
	where 
	$[m,n] = \begin{cases} m & \text{if } m+n \text{ odd} \\ n & \text{if } m+n \text{ even.} \end{cases}$
\end{proposition}

If we sum the Entringer numbers $\ent{n}{k}$ over $k$, we get the count of all alternating permutations on $\{1,\dots , n+1\}$: 
\begin{equation}\label{eq:Entringer sum}
	E_{n+1} = \sum_{k=0}^n \ent{n}{k}.
\end{equation}

\subsection{Augmented Zigzag Posets} 
\label{sub:augmented zigzag posets}

Now we describe a poset whose number of linear extensions is related to expected values of chord lengths in confined polygons.

\begin{definition}\label{def:augmented zigzag}
	Let $Z_{n,i}$ be the poset with ground set $\{0,\dots , n\}$ and partial order given by
	\[
		0 \preceq i, \qquad i \succeq i-1 \preceq i-2 \succeq \dots (\preceq \succeq) 1, \qquad i \succeq i+1 \preceq i+2 \succeq \dots (\preceq \succeq) n,
	\]
	where the direction of the inequalities at the ends of the chain depend on the parity of $i$ and $n-i$: if $i$ is odd, then $2 \preceq 1$, and if $i$ is even, then $2 \succeq 1$; if $n-i$ is odd, then $n-1 \succeq n$, and if $n-i$ is even, then $n-1 \preceq n$. The Hasse diagram for $Z_{n,i}$ is shown in \Cref{fig:Zni}, and the Hasse diagrams for each of the $Z_{4,i}$ are shown in \Cref{fig:Z examples}. We call $Z_{n,i}$ an \emph{augmented zigzag poset}.
\end{definition}

\begin{figure}[htbp]
	\centering
		\begin{tikzpicture}
			\node[my node,label={above:$i$}] (i) at (0,1){};
			\node[my node,label={below:$i-1$}] (i-1) at (-1,0){};
			\node[my node,label={above:$i-2$}] (i-2) at (-2,1){};
			\node[my node,label={below:$i-3$}] (i-3) at (-3,0){};
			\node (i-4) at (-4,1) {};
			\node[my node,label={below:$i+1$}] (i+1) at (1,0){};
			\node[my node,label={above:$i+2$}] (i+2) at (2,1){};
			\node[my node,label={below:$i+3$}] (i+3) at (3,0){};
			\node (i+4) at (4,1) {};
			\node[my node,label={below:$0$}] (0) at (0,0){};
			\draw (i-3) -- (i-2) -- (i-1) -- (i) -- (i+1) -- (i+2) -- (i+3);
			\draw (0) -- (i);
			\draw[dotted] (i-4) -- (i-3);
			\draw[dotted] (i+3) -- (i+4);
		\end{tikzpicture}
	\caption{Hasse diagram for the poset $Z_{n,i}$.}
	\label{fig:Zni}
\end{figure}
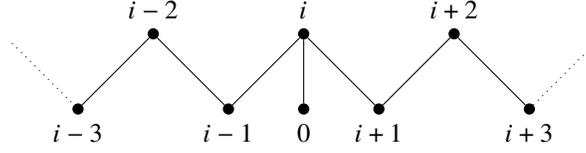

\begin{figure}[htbp]
	\centering
		\begin{tikzpicture}
			\node[my node,label={above:$1$}] (1) at (0,1){};
			\node[my node,label={below:$2$}] (2) at (1,0){};
			\node[my node,label={above:$3$}] (3) at (2,1){};
			\node[my node,label={below:$4$}] (4) at (3,0){};
			\node[my node,label={below:$0$}] (0) at (0,0){};
			\draw (1) -- (2) -- (3) -- (4);
			\draw (0) -- (1);
		\end{tikzpicture}
		\quad
		\begin{tikzpicture}
			\node[my node,label={above:$2$}] (2) at (0,1){};
			\node[my node,label={below:$1$}] (1) at (-1,0){};
			\node[my node,label={below:$3$}] (3) at (1,0){};
			\node[my node,label={above:$4$}] (4) at (2,1){};
			\node[my node,label={below:$0$}] (0) at (0,0){};
			\draw (1) -- (2) -- (3) -- (4);
			\draw (0) -- (2);
		\end{tikzpicture}
		\quad
		\begin{tikzpicture}
			\node[my node,label={above:$3$}] (3) at (0,1){};
			\node[my node,label={below:$2$}] (2) at (-1,0){};
			\node[my node,label={above:$1$}] (1) at (-2,1){};
			\node[my node,label={below:$4$}] (4) at (1,0){};
			\node[my node,label={below:$0$}] (0) at (0,0){};
			\draw (1) -- (2) -- (3) -- (4);
			\draw (0) -- (3);
		\end{tikzpicture}
		\quad
		\begin{tikzpicture}
			\node[my node,label={above:$2$}] (2) at (-2,1){};
			\node[my node,label={below:$1$}] (1) at (-3,0){};
			\node[my node,label={below:$3$}] (3) at (-1,0){};
			\node[my node,label={above:$4$}] (4) at (0,1){};
			\node[my node,label={below:$0$}] (0) at (0,0){};
			\draw (1) -- (2) -- (3) -- (4);
			\draw (0) -- (4);
		\end{tikzpicture}
	\caption{Hasse diagrams for the posets $Z_{4,i}$ for $i=1,2,3,4$.}
	\label{fig:Z examples}
\end{figure}
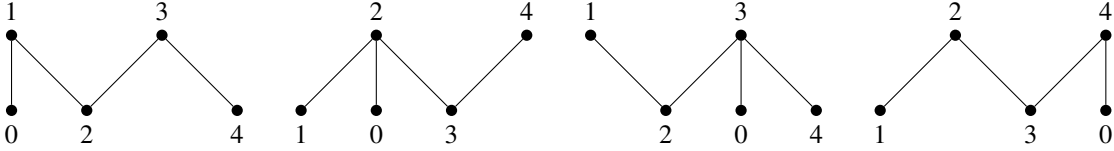

If we remove 0 from the ground set of $Z_{n,i}$, the resulting induced poset is (if $i$ is even) Stanley's \emph{zigzag poset} $Z_n$~\cite[Exercise~3.66]{stanleyEnumerativeCombinatorics2011} or (if $i$ is odd) the opposite $Z_n^{\text{op}}$ of the zigzag poset. Recall that the zigzag poset has ground set $\{1,\dots , n\}$ and cover relations given by $k \preceq k+1$ if $k$ is odd and $k \succeq k+1$ if $k$ is even.

By reversing the order of the labels, $Z_{n,i}$ is equivalent to $Z_{n,n+1-i}$ for all $n$ and all $1 \leq i \leq n$. Also, $Z_{n,1}$ (and hence also $Z_{n,n}$) is equivalent to $Z_{n+1}$.

Let $\Gamma_{n,i}$ be the cover graph of $Z_{n,i}$ (i.e., ignoring all order information from the Hasse diagram). Then $\Gamma_{n,1}$, $\Gamma_{n,2}$, and $\Gamma_{n,3}$ are copies of the Dynkin diagrams $A_{n+1}$, $D_{n+1}$, and $E_{n+1}$, respectively.

\begin{notation}
	For a poset $P = (X,\preceq)$, let $e(P)$ denote the number of linear extensions of $P$; that is, the number of total orders on $X$ compatible with $\preceq$.
\end{notation}

\begin{example}
	Let $Z_n$ be the zigzag poset. Then $e(Z_n) = E_n$~\cite[Exercise~3.66(c)]{stanleyEnumerativeCombinatorics2011}. Hence, $e(Z_{n,1}) = E_{n+1} = e(Z_{n,n})$.
\end{example}

\subsection{The Number of Linear Extensions of $Z_{n,i}$}
\label{sub:linear extensions of zni}

In \Cref{sec:chord lengths}, we will relate expected values of the chord lengths $d_i$ to the $e(Z_{n,i})$, so it would be nice to be able to compute these numbers. In general, the problem of computing the number of linear extensions of a poset is $\# P$-complete~\cite{linialHardEnumerationProblems1986,brightwellCountingLinearExtensions1991}, but Atkinson~\cite{atkinsonComputingNumberLinear1990} showed that there is an $O(n^2)$ algorithm for computing the number of linear extensions of any poset of size $n$ whose cover graph is a tree. Since $\Gamma_{n,i}$ is a tree, Atkinson's result applies and, by unpacking the algorithm a bit, we will get a formula for $e(Z_{n,i})$ in terms of Entringer numbers, for which we have the recurrence~\eqref{eq:Entringer recurrence}, the explicit formula~\eqref{eq:entringer formula} in terms of Euler numbers, and the exponential generating function given in \Cref{prop:entringer egf}. 

To explain Atkinson's algorithm, we introduce some terminology.

\begin{definition}\label{def:spectrum}
	Suppose $P = (X,\preceq)$ is a poset, where $X$ is a finite set of cardinality $m$. If $\alpha \in X$, the $\alpha$-spectrum of $P$ is the sequence $(\lambda_1, \dots , \lambda_m)$ where $\lambda_i$ is the number of linear extensions of $P$ so that $\alpha$ has rank $i$ in the total ordering.
\end{definition}

\begin{proposition}[{Atkinson~\cite[Lemma~1]{atkinsonComputingNumberLinear1990}}]\label{prop:splitting}
	Let $P = (X, \preceq_P)$ and $Q = (Y, \preceq_Q)$ be disjoint posets of sizes $p$ and $q$, respectively. Let $\alpha \in X$ and $\beta \in Y$, let $(\lambda_1, \dots , \lambda_p)$ be the $\alpha$-spectrum of $P$, and let $(\mu_1, \dots , \mu_q)$ be the $\beta$-spectrum of $Q$. Consider the poset $R = (X \cup Y, \preceq)$ whose covering relations are those of $P$ and $Q$ together with the relation $\alpha \preceq \beta$. Then the $\alpha$-spectrum of $R$ is $(\nu_1, \dots , \nu_{p+q})$ where
	\[
		\nu_k = \sum_{i=\max\{1,k-q\}}^{\min\{p,k\}} \lambda_i \binom{k-1}{i-1}\binom{p+q-k}{p-i} \sum_{j=k-i+1}^q\mu_j.
	\]
\end{proposition}

Then Atkinson's algorithm consists of splitting the cover graph of the poset (which, again, is a tree) at some vertex and recursively computing the spectra of each side, then computing the $\nu_k$ with \Cref{prop:splitting}  and summing over $k$.

Consider $Z_{m+1}$, the zigzag poset. This is the poset on $\{1,\dots, m+1\}$ with $1 \preceq 2 \succeq 3 \preceq \dots$. Then the $1$-spectrum of $Z_{m+1}$ consists of counts of the number of reverse-alternating permutations $\tau$ of $\{1, \dots , m+1\}$ with $\tau(1)$ specified. Since the reverse map $\tau \mapsto \overline{\tau}$ defined by $\overline{\tau}(i) := m+2-\tau(i)$ sends reverse-alternating permutations to alternating permutations, these are counted by Entringer numbers:

\begin{lemma}\label{lem:zigzag 1-spectrum}
	The 1-spectrum of $Z_{m+1}$ is
	\[
		\left(\ent{m}{m}, \dots , \ent{m}{0}\right);
	\]
	that is, the $r$th term in the 1-spectrum is $\ent{m}{m+1-r}$.
\end{lemma}

Summing gives the number of linear extensions of $Z_{m+1}$, which we know is the Euler number $E_{m+1}$, so this just recapitulates the identity~\eqref{eq:Entringer sum}.

%
%
%
%
%

Now we will compute the 2-spectrum of $Z_{m+1}$:

\begin{proposition}\label{prop:zigzag 2-spectrum}
	The $r$th term in the 2-spectrum of $Z_{m+1}$ is
	\[
		(r-1)\ent{m-1}{r-2}.
	\]
\end{proposition}

\begin{proof}
	Equivalently, we are counting the number of reverse-alternating permutations $\tau$ on $\{1, \dots , m+1\}$ with $\tau(2) = r$. But, by dropping $1$ from the domain and $\tau(1)$ from the range and then shifting indices and values down to fill these gaps, we get an alternating permutation $\widetilde{\tau}$ on $\{1, \dots , m\}$ with $\widetilde{\tau}(1) = r-1$. There are $\ent{m-1}{r-2}$ such induced permutations and the map $\tau \mapsto \widetilde{\tau}$ is $(r-1)$-to-$1$ since $\widetilde{\tau}$ can be uniquely extended to a valid $\tau$ once we have chosen $\tau(1) < \tau(2) = r$. This proves the result.
\end{proof}

Summing \Cref{prop:zigzag 2-spectrum} over $r$ once again gives the number of linear extensions of $Z_{m+1}$, namely $E_{m+1}$. So, re-indexing with $n=m-1$ and $k=r-2$, this implies the following result which we have not been able to find in the literature:

\begin{corollary}\label{cor:entringer sum}
	For all $n\geq 1$,
	\begin{equation}\label{eq:entringer sum}
		\sum_{k=1}^{n} (k+1) \ent{n}{k} = E_{n+2}.
	\end{equation}
\end{corollary}

As an immediate corollary, we can compute the expected value of $\tau(1)$ when $\tau$ is an alternating permutation:

\begin{corollary}\label{cor:expected first entry}
	The average of $\tau(1)$ over all alternating permutations on $\{1,\dots , n\}$ is
	\[
		\mathbb{E}[\tau(1) \,|\, \tau \in AP_n] = \frac{E_{n+1}}{E_n} \sim \frac{2}{\pi}(n+1).
	\]
\end{corollary}

\begin{proof}
	Since there are $E_n$ alternating permutations,
	\[
		\mathbb{E}[\tau(1) \,|\, \tau \in AP_n] = \frac{1}{E_n}\; \smashoperator[r]{\sum_{\tau \in AP_n}}\; \tau(1) = \frac{1}{E_n} \sum_{k=1}^{n-1} (k+1) \ent{n-1}{k} = \frac{E_{n+1}}{E_n},
	\]
	where we used \Cref{cor:entringer sum} for the last equality. The asymptotic estimate then follows from~\eqref{eq:euler asymptotics}.
\end{proof}

Alternating permutations have $\left\lfloor \frac{n}{2} \right\rfloor$ descents, and Conger~\cite{congerRefinementEulerianNumbers2010} proved that the expectation of $\tau(1)$ over all permutations with $\left\lfloor \frac{n}{2} \right\rfloor$ descents is exactly $\left\lfloor \frac{n}{2} \right\rfloor + 1$. This is slightly smaller than the expectation $\frac{E_{n+1}}{E_n} \sim \frac{2}{\pi}(n+1)$ for alternating descents, which makes sense: an alternating permutation \emph{must} start with a descent, so the first entry should tend to be slightly larger than average.

We now have all the necessary tools in place to compute the spectrum of $Z_{n,i}$, which will lead to a formula for $e(Z_{n,i})$.

\begin{proposition}\label{prop:augmented zigzag spectrum}
	Assume $1<i$. Letting $\alpha = i-1$, the $r$th term in the $\alpha$-spectrum of $Z_{n,i}$ is
	\[
		\sum_{j=\max\{1,r-(n-i+2)\}}^{\min\{i-1,r\}} \ent{i-2}{i-1-j} \binom{r-1}{j-1}\binom{n+1-r}{i-1-j}\sum_{k = r-j+1}^{n-i+2}(k-1)\ent{n-i}{k-2}.
	\]
\end{proposition}

\begin{proof}
	The poset $Z_{n,i}$ has elements $\{0,1,\dots , n\}$. We will compute the spectrum using \Cref{prop:splitting} with $P$ being the induced poset on $\{1, \dots , i-1\}$, and $Q$ being the induced poset on $\{0,i,\dots , n\}$ (this is where we use the assumption $1 < i$, since otherwise $P$ is empty). Thus, $\beta = i$.
	
	After reversing the order of the indices, $P$ is just a copy of $Z_{i-1}$, and we know from \Cref{lem:zigzag 1-spectrum} that the $1$-spectrum of $Z_{i-1}$ (and hence the $\alpha$-spectrum of $P$) has $j$th term
	\[
		\lambda_j = \ent{i-2}{i-1-j}.
	\]
	
	On the other hand, $Q$ is a copy of $Z_{n-i+2}$, where $0$ is identified with the element $1$ in $Z_{n-i}$ and $\beta=i$ is identified with 2. So then \Cref{prop:zigzag 2-spectrum} computes the $\beta$-spectrum of $Q$, which has $k$th term
	\[
		\mu_k = (k-1)\ent{n-i}{k-2}.
	\]
	Plugging the $\lambda_j$ and $\mu_k$ into \Cref{prop:splitting} gives the result.
\end{proof}

Finally, we can sum terms to get the number of linear extensions of $Z_{n,i}$ in terms of Entringer numbers:

\begin{corollary}\label{cor:e(Z_{n,i})}
	\[
		e(Z_{n,i}) = \sum_{r=1}^{n+1}\sum_{j=\max\{1,r-(n-i+2)\}}^{\min\{i-1,r\}} \ent{i-2}{i-1-j} \binom{r-1}{j-1}\binom{n+1-r}{i-1-j}\sum_{k = r-j+1}^{n-i+2}(k-1)\ent{n-i}{k-2}.
	\]
	For $n \leq 10$, the values of $e(Z_{n,i})$ are shown in \Cref{tab:linear extensions}.
\end{corollary}

\begin{table}
	\centering
	\begin{tabular}{l@{\hskip 0.25in}rrrrrrrrrr}
		\multicolumn{11}{c}{$e(Z_{n,i})$} \\
		\toprule
	  \diagbox{$n$}{$i$} & 1 & 2 & 3 & 4 & 5 & 6 & 7 & 8 & 9 & 10 \\
	   \midrule
	 1 & 1 &  &  &  &  &  &  &  &  &  \\
	 2 & 2 & 2 &  &  &  &  &  &  &  &  \\
	 3 & 5 & 6 & 5 &  &  &  &  &  &  &  \\
	 4 & 16 & 18 & 18 & 16 &  &  &  &  &  &  \\
	 5 & 61 & 70 & 66 & 70 & 61 &  &  &  &  &  \\
	 6 & 272 & 310 & 298 & 298 & 310 & 272 &  &  &  &  \\
	 7 & 1385 & 1582 & 1511 & 1540 & 1511 & 1582 & 1385 &  &  &  \\
	 8 & 7936 & 9058 & 8670 & 8780 & 8780 & 8670 & 9058 & 7936 &  &  \\
	 9 & 50521 & 57678 & 55168 & 55986 & 55630 & 55986 & 55168 & 57678 & 50521 &  \\
	 10 & 353792 & 403878 & 386394 & 391846 & 390176 & 390176 & 391846 & 386394 & 403878 & 353792 \\
	 \bottomrule
	\end{tabular}
	\caption{The number of linear extensions of $Z_{n,i}$ for $1 \leq i \leq n \leq 10$. Up to index shifts (and dropping some leading terms), the first column is \seqnum{A000111}, the second column is \seqnum{A131281}, and the third column is \seqnum{A131611} (or \seqnum{A131656}) in OEIS~\cite{oeis}. The values in these three columns agree (as \Cref{prop:principal numbers} says they must) with the principal numbers of $A_{n+1}$, $D_{n+1}$, and $E_{n+1}$ computed by Sano~\cite{sanoPrincipalNumbersSaito2007}.}
	\label{tab:linear extensions}
\end{table}

\subsection{A Recursive Formula for $e(Z_{n,i})$} 
\label{sub:recursive formula}

As an alternative to \Cref{cor:e(Z_{n,i})}, we now construct a recursive formula for $e(Z_{n,i})$ using ideas of Saito~\cite{saitoPrincipal$Gamma$coneTree2007}.

\begin{definition}[Saito~\cite{saitoPrincipal$Gamma$coneTree2007}]\label{def:principal numbers}
	Let $\Gamma$ be a tree. Any orientation $\mathfrak{o}$ on $\Gamma$ determines a poset $P_{\mathfrak{o}}$ whose ground set consists of the vertices of $\Gamma$ and whose cover relations are determined by the orientation. The \emph{principal number} of $\Gamma$ is
	\[
		\sigma(\Gamma) := \max \{ e(P_{\mathfrak{o}}) : \mathfrak{o} \text{ is an orientation on } \Gamma \}.
	\]
\end{definition}

Saito proved that the orientations on a tree $\Gamma$ which maximize $e(P_{\mathfrak{o}})$ are the two bipartite (or principal) orientations, namely those with no directed path of length 2. In particular, one of the principal orientations on $\Gamma_{n,i}$ induces the poset $Z_{n,i}$, proving the following:

\begin{proposition}\label{prop:principal numbers}
	For $1 \leq i \leq n$,
	\[
		e(Z_{n,i}) = \sigma(\Gamma_{n,i}).
	\]
\end{proposition}

Saito gave a recursive expression for $\sigma(\Gamma)$ which was refined by Sano~\cite{sanoPrincipalNumbersSaito2007} as follows:

\begin{proposition}[{Sano~\cite[Corollary~2.2]{sanoPrincipalNumbersSaito2007}}]\label{prop:principal number recurrence}
	For a tree $\Gamma$ with vertex set $V$,
	\[
		\sigma(\Gamma) = \frac{1}{2} \sum_{v \in V} \left[ \prod_{k=1}^{c_v} \binom{|\Gamma|-1-\sum_{i=1}^{k-1} |\Gamma_v^{(i)}|}{|\Gamma_v^{(k)}|}\prod_{k=1}^{c_v} \sigma(\Gamma_v^{(k)})\right],
	\]
	where the $\Gamma_v^{(k)}$ with $k=1, \dots , c_v$ are the connected components of the graph $\Gamma \backslash\{v\}$.
\end{proposition}

Since we know from \Cref{prop:principal numbers} that $e(Z_{n,i}) = \sigma(\Gamma_{n,i})$, a straightforward but tedious calculation yields a recursive formula for $e(Z_{n,i})$:

\begin{corollary}\label{cor:linear extension recurrence}
	For $1 < i \leq n$,
	\begin{multline}\label{eq:linear extension recurrence}
		e(Z_{n,i}) = \frac{1}{2}\left[ E_n + e(Z_{n-1,i-1})+ \sum_{j=2}^{i-1} \binom{n}{j-1} E_{j-1}e(Z_{n-j,i-j}) + n \binom{n-1}{i-1}E_{i-1}E_{n-i} \right. \\ \left.+ \sum_{\ell=i+1}^{n}\binom{n}{\ell} e(Z_{\ell-1,i})E_{n-\ell}\right].
	\end{multline}
\end{corollary}

It is not clear what $Z_{m,i}$ should mean when $m < i$, but---following Sano's example~\cite{sanoPrincipalNumbersSaito2007}---it is possible to assign values of $e(Z_{m,i})$ for $m < i$ so that the recursive expression in \Cref{cor:linear extension recurrence} simplifies to
\begin{equation}\label{eq:simplified recurrence}
	e(Z_{n,i}) = \frac{1}{2}\sum_{\ell = 1}^n \binom{n}{\ell} e(Z_{\ell-1,i})E_{n-\ell}
\end{equation}
for $n \geq i$. Notice that the coefficient of $E_{n-i}$ in \eqref{eq:linear extension recurrence} is $n \binom{n-1}{i-1}E_{i-1} = \binom{n}{i}iE_{i-1}$, so we will automatically assign $e(Z_{i-1,i}) := i E_{i-1}$. Moreover, expanding the $e(Z_{n-j,i-j})$ terms in the first sum in \eqref{eq:linear extension recurrence}, we see that there is no $E_{n-(i-1)}$ term, so we will assign $e(Z_{i-2,i}) :=0$.

Therefore, (cf.~\cite{millarNewOperationSequences1996}), the exponential generating function $f_i$ of the sequence $e(Z_{0 , i}),e(Z_{1 , i}),\dots$ has the form
\begin{equation}\label{eq:egf ode}
	f_i'(x) = \frac{1}{2} f_i(x)(\sec x + \tan x) + p_i(x),
\end{equation}
where $p_i$ is a polynomial of degree $i-1$. Using $e(Z_{i-2,i}) =0$, $e(Z_{i-1,i}) = i E_{i-1}$, and $e(Z_{i+j,i}) = e(Z_{i+j,j+1})$ for $j=0, \dots , i-2$, we can deduce $i+1$ initial conditions
\begin{equation}\label{eq:egf ode ics}
	f_i^{(i-1)}(0) = 0, \quad f_i^{(i)}(0) = i E_{i-1}, \quad \text{and} \quad f_i^{(i+j+1)}(0) = e(Z_{i+j,j+1}) \text{ for }j=0, \dots , i-2.
\end{equation}
These are enough to allow us to uniquely solve \eqref{eq:egf ode} for the exponential generating function $f_i$ for any particular $i$, though we have not been able to find a general expression for $p_i(x)$.

\subsection{Generalized Entringer Numbers} 
\label{sub:generalized_entringer_numbers}

As our final excursion in pure combinatorics, we define generalized Entringer numbers and show that they satisfy an analog of \Cref{cor:entringer sum}. This will then provide the connection in \Cref{sec:chord lengths} between the expected values of the chord lengths $d_i$ and the $e(Z_{n,i})$.

\begin{definition}\label{def:generalized entringer}
	Let $1 \leq i \leq n+1$, $0 \leq k \leq n$, and define the \emph{generalized Entringer number} $\gent{n}{k}{i}$ as follows. If $i$ is odd, let $\gent{n}{k}{i}$ be the number of alternating permutations $\tau \in S_{n+1}$ with $\tau(i) = k+1$; if $i$ is even, let $\gent{n}{k}{i}$ be the number of reverse-alternating permutations $\tau \in S_{n+1}$ with $\tau(i) = k+1$.
\end{definition}

Note that $\gent{n}{k}{1} = \ent{n}{k}$, so this is a generalization of the Entringer numbers.

While the generalized Entringer numbers do not seem to be well-studied, so we don't know much about them, they certainly satisfy the analog of~\eqref{eq:Entringer sum}: for any $1 \leq i \leq n+1$,
\begin{equation}\label{eq:generalized Entringer sum}
	E_{n+1} = \sum_{k=0}^n \gent{n}{k}{i}.
\end{equation}

Our main reason for defining the generalized Entringer numbers is the following generalization of \Cref{cor:entringer sum}:

\begin{theorem}\label{thm:generalized entringer sum}
	\[
		\sum_{k=1}^n k \gent{n-1}{k-1}{i} = e(Z_{n,i}).
	\]
\end{theorem}

\begin{proof}
	Suppose $\sigma$ is a linear extension of $Z_{n,i}$; that is, $\sigma: \{0,\dots , n\} \to \{0, \dots , n\}$ is an order-preserving bijection: $\sigma(i) \leq \sigma(j)$ if $i \preceq j$. If we remove 0 from the ground set of $Z_{n,i}$ and define $\widetilde{\sigma}:\{1, \dots , n\} \to \{1, \dots , n\}$ by
	\[
		\widetilde{\sigma}(i) = \begin{cases} \sigma(i)+1 & \text{ if } \sigma(i) < \sigma(0) \\ \sigma(i) & \text{ else,} \end{cases}
	\]
	then $\widetilde{\sigma}$ is a linear extension of the zigzag poset $Z_n$ (if $i$ is odd) or its opposite $Z_n^{\text{op}}$ (if $i$ is even).

	In other words, the mapping $R: \sigma \mapsto \widetilde{\sigma}$ gives a surjective map from the set of linear extensions of $Z_{n,i}$ to the set of linear extensions of either $Z_n$ or $Z_n^{\text{op}}$. Of course, $R$ is not injective; for example, if $n=2$ and $i=1$, both $\sigma_1: (0,1,2) \mapsto (0,2,1)$ and $\sigma_2: (0,1,2) \mapsto (1,2,0)$ map to the same alternating permutation $\widetilde{\sigma}: (1,2) \mapsto (2,1)$.

	So how many $\sigma$ are in the preimage of a given $\widetilde{\sigma}$? The only freedom is in the choice of $\sigma(0)$, and the only constraint is that $\sigma(0) < \sigma(i) = \widetilde{\sigma}(i)$, so there are $\widetilde{\sigma}(i)$ possible choices for $\sigma(0)$, namely $0, \dots , \widetilde{\sigma}(i) - 1$.

	Therefore, each of the $\gent{n-1}{k-1}{i}$ (alternating or reverse-alternating) permutations $\widetilde{\sigma}$ in the range of $R$ with $\widetilde{\sigma}(i) = k$ have $k$ preimages, so
	\[
		\sum_{k=1}^n k \gent{n-1}{k-1}{i}
	\]
	must equal the total number of elements in the domain of $R$, which is exactly $e(Z_{n,i})$.
\end{proof}

Unlike \Cref{cor:e(Z_{n,i}),cor:linear extension recurrence}, \Cref{thm:generalized entringer sum} does not seem to be of much practical use for computing $e(Z_{n,i})$, but it will be an essential piece of our proof of \Cref{thm:chord lengths}, which relates the expected values of the~$d_i$ to the $e(Z_{n,i})$.

Before we get to that, one last combinatorial observation: we can use the same proof as in \Cref{cor:expected first entry} to  show that \Cref{thm:generalized entringer sum} implies the following:

\begin{corollary}\label{cor:expected ith entry}
	If $i$ is odd, the average of $\tau(i)$ over all alternating permutations on $\{1,\dots, n\}$ is
	\[
		\mathbb{E}[\tau(i) \,|\, \tau \in AP_n] = \frac{e(Z_{n,i})}{E_n}.
	\]

	If $i$ is even, the average of $\tau(i)$ over all reverse alternating permutations on $\{1,\dots, n\}$ is
	\[
		\mathbb{E}[\tau(i) \,|\, \tau \in RP_n] = \frac{e(Z_{n,i})}{E_n}.
	\]
\end{corollary}

Similarly, if $\tau$ is an alternating permutation, then its reverse $\overline{\tau}$ is reverse-alternating, so linearity of expectations implies
\[
	\mathbb{E}[\tau(i)\,|\, \tau \in AP_n] = n+1 - \mathbb{E}[\tau(i)\,|\, \tau \in RP_n],
\]
which proves:

\begin{corollary}\label{cor:expected ith entry pt 2}
	If $i$ is even, the average of $\tau(i)$ over all alternating permutations on $\{1,\dots, n\}$ is
	\[
		\mathbb{E}[\tau(i) \,|\, \tau \in AP_n] = n+1 - \frac{e(Z_{n,i})}{E_n}.
	\]

	If $i$ is odd, the average of $\tau(i)$ over all reverse alternating permutations on $\{1,\dots, n\}$ is
	\[
		\mathbb{E}[\tau(i) \,|\, \tau \in RP_n] = n+1 - \frac{e(Z_{n,i})}{E_n}.
	\]
\end{corollary}


\section{Expected Chord Lengths}
\label{sec:chord lengths}

We return to confined polygons and consider the question: what is the expected value of $d_i = |v_{i+2} - v_1|$ as a function on the space $\Polhat(n;1)$ of equilateral polygons in rooted spherical confinement of radius 1?

Since the vector $(d_1, \dots , d_{n-3})$ is distributed according to (normalized) Lebesgue measure on $\mathcal{P}_n(1)$,
\[
	\mathbb{E}[(d_1, \dots , d_{n-3})] = (\mathbb{E}[d_1], \dots , \mathbb{E}[d_{n-3}])
\]
is exactly the center of mass of $\mathcal{P}_n(1)$. If we want to compute the center of mass of $\mathcal{P}_n(1)$, it is enough to compute the center of mass $\overline{x}$ of the order polytope $\mathcal{O}_{n-3} = \varphi(\mathcal{P}_n(1))$, since then
\[
	\mathbb{E}[(d_1, \dots , d_{n-3})] = \varphi(\overline{x})
\]
(recall that $\varphi$ is an involution, so $\varphi = \varphi^{-1}$). The map $\varphi$ fixes coordinates of odd index, so we record the following observation:

\begin{lemma}\label{lem:com odd coords}
	For $i$ odd,
	\[
		\mathbb{E}[d_i] = \overline{x}_i.
	\]
\end{lemma}

Similarly, the map $\varphi \circ \psi$ fixes coordinates of even index. This map sends $\mathcal{P}_n(1)$ to the polytope
\[
	\mathcal{U}_{n-3} := \{(y_1, \dots , y_{n-3}) \in [0,1]^{n-3} : y_1 \leq y_2 \geq y_3 \leq y_4 \geq \dots \},
\]
which is the order polytope of $Z_{n-3}^{\text{op}}$, the opposite of the zigzag poset. Analogously to $\mathcal{U}_{n-3}$, any map $\sigma: \mathcal{U}_{n-3} \to \{1, \dots , n-3\}$ is order-preserving if and only if the permutation $i \mapsto \sigma(y_i)$ is reverse-alternating, and hence $\mathcal{U}_{n-3}$ decomposes into orthoschemes
\begin{equation}\label{eq:orthoscheme2}
	\Delta_\tau = \{(y_1, \dots , y_{n-3}) \in [0,1]^{n-3}: y_{\tau^{-1}(1)} \leq y_{\tau^{-1}(2)} \leq \dots \leq y_{\tau^{-1}(n-3)} \} \subset \mathcal{U}_{n-3}
\end{equation}
indexed by reverse-alternating permutations.

Let $\overline{y}$ be the center of mass of $\mathcal{U}_{n-3}$. Analogously to \Cref{lem:com odd coords} we have
\begin{lemma}\label{lem:com even coords}
	For $i$ even,
	\[
		\mathbb{E}[d_i] = \overline{y}_i.
	\]
\end{lemma}

These interpretations are not obviously helpful, since computing the centroid of a convex polytope is $\#P$-hard, even when the polytope is an order polytope~\cite{rademacherApproximatingCentroidHard2007}.
However, we have unimodular triangulations of $\mathcal{O}_{n-3}$ and $\mathcal{U}_{n-3}$ into the orthoschemes $\Delta_\tau$ defined in \eqref{eq:orthoscheme definition} and \eqref{eq:orthoscheme2}, so the center of mass of $\mathcal{O}_{n-3}$ is the mean of the centers of masses of the orthoschemes:
\begin{equation}\label{eq:com1}
	\overline{x} = \frac{1}{E_{n-3}}\; \smashoperator[r]{\sum_{\tau \in AP_{n-3}}}\; m(\Delta_\tau),
\end{equation}
where $m(\Delta_\tau)$ is the center of mass of $\Delta_\tau$, and similarly
\begin{equation}\label{eq:com2}
	\overline{y} = \frac{1}{E_{n-3}}\; \smashoperator[r]{\sum\limits_{\tau \in RP_{n-3}}}\; m(\Delta_\tau).
\end{equation}

In turn, the center of mass of a simplex is simply the center of mass of its vertices (see Krantz, McCarthy, and Parks~\cite{krantzGeometricCharacterizationsCentroids2006} for much more on centroids of simplices), so
\begin{equation}\label{eq:orthoscheme com}
	m(\Delta_\tau) = \frac{1}{n-2}\; \smashoperator[r]{\sum_{v \text{ a vertex of } \Delta_\tau}}\; v.
\end{equation}

A vertex $v = (v_1, \dots , v_{n-3})$ of an orthoscheme $\Delta_\tau$ is determined by assignments of 0s and 1s to the coordinates that satisfy the system of inequalities
\[
	v_{\tau^{-1}(1)} \leq v_{\tau^{-1}(2)} \leq \dots \leq v_{\tau^{-1}(n-3)}.
\]
Clearly, then, only one vertex (namely $(1,\dots , 1)$) can have $v_{\tau^{-1}(1)} = 1$, only two can have $v_{\tau^{-1}(2)} = 1$, etc. In other words:

\begin{lemma}\label{lem:orthoscheme vertices}
	Let $\tau \in S_{n-3}$ be a permutation, and let $\Delta_\tau$ be the corresponding orthoscheme. Then $\tau(i)$ of the vertices of $\Delta_\tau$ have $i$th coordinate equal to 1.
\end{lemma}

Combining this lemma with \eqref{eq:com1} and \eqref{eq:orthoscheme com} yields
\[
	\overline{x}_i = \frac{1}{(n-2)E_{n-3}} \; \smashoperator[r]{\sum_{\tau \in AP_{n-3}}}\; \tau(i).
\]
In turn, if $i$ is odd then there are $\gent{n-4}{k-1}{i}$ alternating permutations with $\tau(i) = k$, so we can re-write the above sum as
\[
	\overline{x}_i = \frac{1}{(n-2)E_{n-3}} \sum_{k=1}^{n-3} k \gent{n-4}{k-1}{i} = \frac{e(Z_{n-3,i})}{(n-2)E_{n-3}}
\]
by \Cref{thm:generalized entringer sum}. Similarly, when $i$ is even we have
\[
	\overline{y}_i = \frac{1}{(n-2)E_{n-3}} \sum_{k=1}^{n-3} k \gent{n-4}{k-1}{i} = \frac{e(Z_{n-3,i})}{(n-2)E_{n-3}}.
\]

Combined with \Cref{lem:com odd coords,lem:com even coords}, this tells us the expected values of the $d_i$:

\begin{theorem}\label{thm:chord lengths}
	The expected value of the $i$th chordlength $d_i$ of a random confined equilateral polygon is
	\[
		\mathbb{E}[d_i] = \frac{e(Z_{n-3,i})}{(n-2)E_{n-3}}.
	\]
	Note that the denominator is \seqnum{A065619} in OEIS~\cite{oeis}.
\end{theorem}

For example, we know that $e(Z_{n-3,1}) = e(Z_{n-3,n-3}) = E_{n-2}$, so we see that average first and last chord lengths are
\[
	\mathbb{E}[d_1] = \mathbb{E}[d_{n-3}] = \frac{E_{n-2}}{(n-2)E_{n-3}} \sim \frac{2}{\pi} \approx 0.6366
\]
using the asymptotic expansion~\eqref{eq:euler asymptotics}.

More generally, we can use Sano's exponential generating functions for $e(Z_{n,1})$, $e(Z_{n,2})$, and $e(Z_{n,3})$~\cite{sanoPrincipalNumbersSaito2007} and straightforward generalizations to slightly bigger $i$ given by using \Cref{cor:linear extension recurrence} or solving \eqref{eq:egf ode} and \eqref{eq:egf ode ics} to prove the following asymptotics for a few of the $\mathbb{E}[d_i]$:

\begin{corollary}\label{prop:chordlength asymptotics}
	The first few chord lengths have the following asymptotics as the number of edges $n \to \infty$:
	\begin{align*}
		\mathbb{E}[d_1] & \sim \frac{2}{\pi} \approx 0.6366 \\
		\mathbb{E}[d_2] & \sim 2-\frac{4}{\pi} \approx 0.7267 \\
		\mathbb{E}[d_3] & \sim \frac{\pi}{4} - 2 + \frac{6}{\pi} \approx 0.6952 \\
		\mathbb{E}[d_4] & \sim \frac{\pi^2}{12} - \frac{\pi}{2} + 4 - \frac{8}{\pi} \approx 0.7051 \\
		\mathbb{E}[d_5] & \sim \frac{5\pi^3}{192} - \frac{\pi^2}{6} + \frac{3\pi}{4} - 4 + \frac{10}{\pi} \approx 0.7018 \\
		\mathbb{E}[d_6] & \sim \frac{\pi^4}{120} - \frac{5\pi^3}{96} + \frac{\pi^2}{4} - \pi + 6 - \frac{12}{\pi} \approx 0.7029 \\
		\mathbb{E}[d_7] & \sim \frac{61 \pi^5}{23040} - \frac{\pi^4}{60} + \frac{5\pi^3}{64} - \frac{\pi^2}{3} + \frac{5\pi}{4} - 6 + \frac{14}{\pi} \approx 0.7025.
	\end{align*}
\end{corollary}

This suggests the following conjecture:
\begin{conjecture}\label{conj:chordlength asymptotics}
	For $k \geq 2$,
	\[
		\mathbb{E}[d_k] \sim \sum_{\ell=1}^{k-2} (-1)^{\ell+1} \frac{\ell E_{k-\ell}}{2^{k-\ell-1}(k-\ell)!} \pi^{k-\ell-1} + (-1)^k \left(2 \left\lfloor \frac{k}{2} \right\rfloor - \frac{2k}{\pi}\right).
                                                                                                    	\]
\end{conjecture}

Even for relatively small $i$, these values are rapidly approaching the overall limiting value $\frac{1}{2} + \frac{2}{\pi^2} \approx 0.7026$ from \Cref{cor:asymptotic chordlength} and the convergence in $n$ is also quite rapid; see \Cref{fig:chordlength asymptotics}.

\begin{figure}[t]
	\centering
		\includegraphics[width=5in]{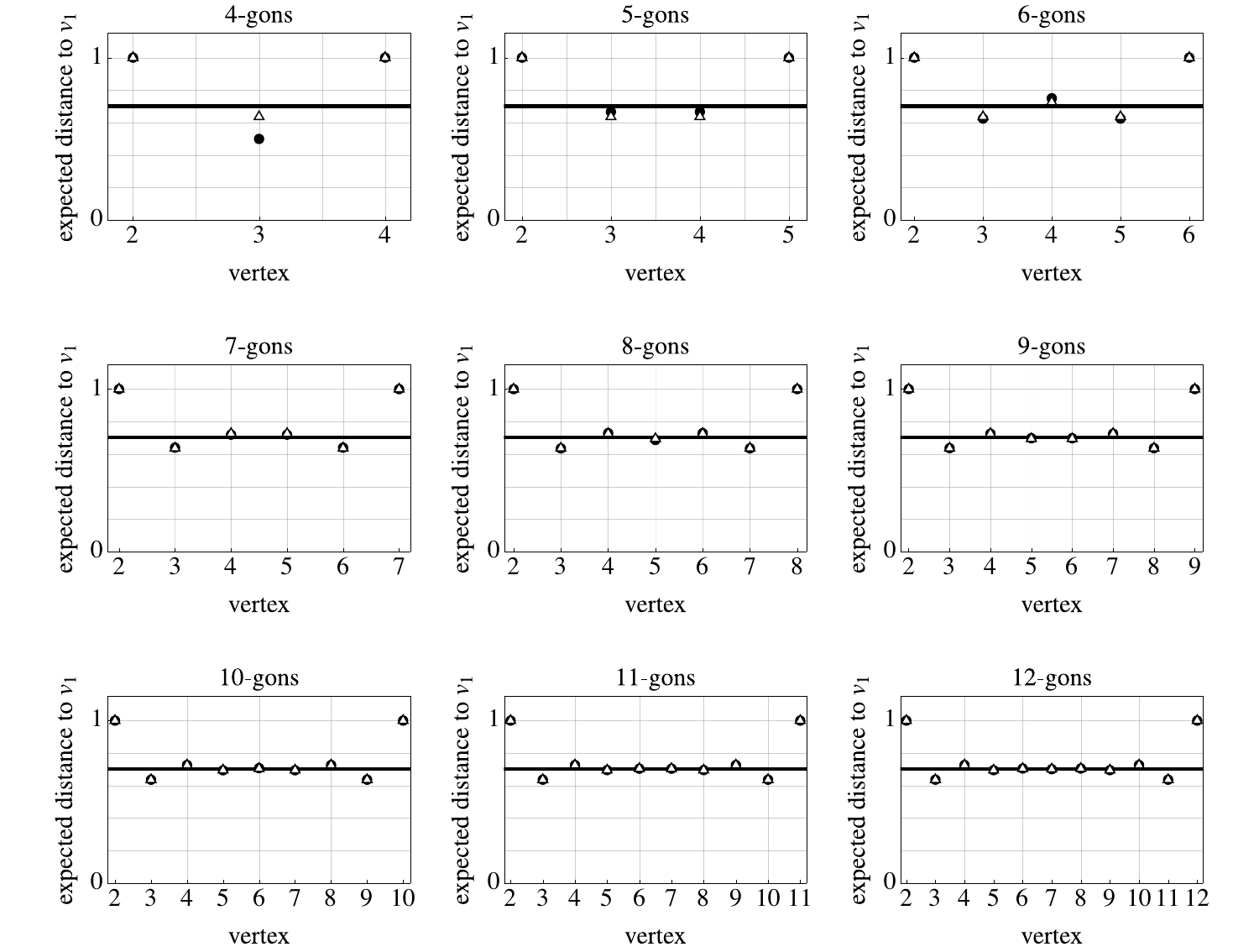}
	\caption{Expected values of the distances from $v_i$ to $v_1$ in confined equilateral $n$-gons (marked with the solid dot $\bullet$) compared to the asymptotic values from \Cref{prop:chordlength asymptotics} (marked with the open triangle $\triangle$) and the value $\frac{1}{2} + \frac{2}{\pi^2} \approx 0.7026$ from \Cref{cor:asymptotic chordlength} (solid line). Note that $|v_2 - v_1|$ and $|v_n - v_1|$ are always exactly~1.}
	\label{fig:chordlength asymptotics}
\end{figure}


\section{Numerical Experiments on Average Turning Angle}
\label{sec:numerics}

	\begin{figure}[t]
		\centering
			\includegraphics[height=1.7in]{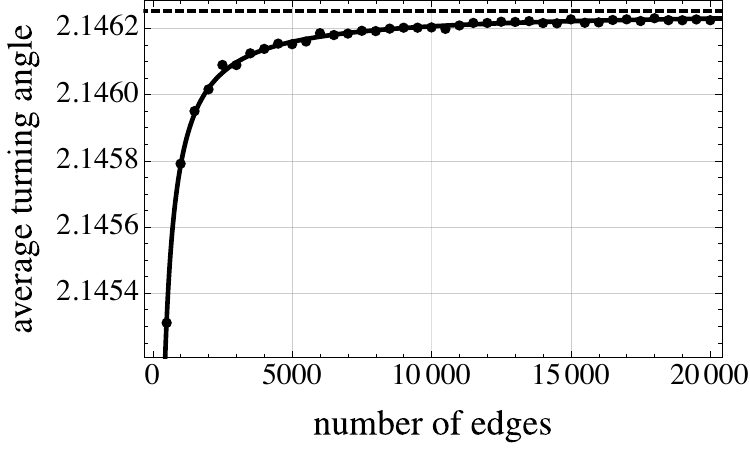} \qquad
			\includegraphics[height=1.7in]{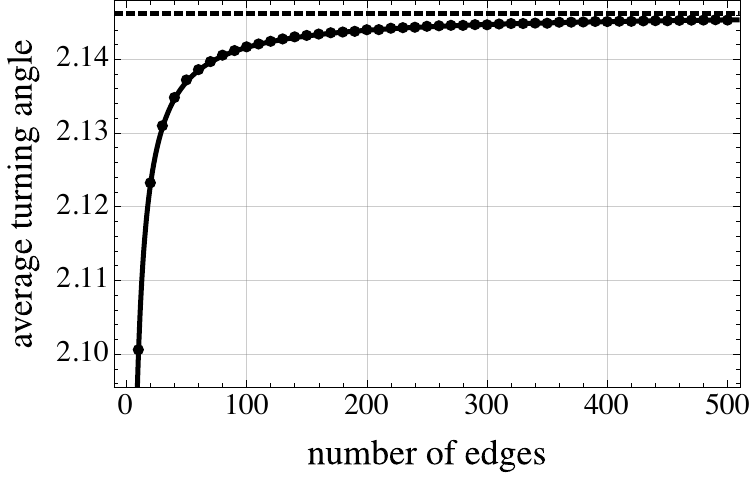}
		\caption{Average turning angle for 1,000,000 random $n$-gons in $\Polhat(n;1)$ for a range of $n$. \emph{Left:} $n$ from from 500 to 20,000 in steps of 500; \emph{Right:} $n$ from 10 to 500 in steps of 10. In both plots, the solid curve is the graph of the fitted function $f(n) = 2.14625-\frac{0.46742}{n}$, and the dashed line is the limiting value $2.14625 \approx \frac{\pi}{2} + 0.57545$ of $f(n)$.}
		\label{fig:total curvature}
	\end{figure}

	We now move beyond the realm of things we can prove and use CPOP to explore the average total curvature of tightly confined equilateral polygons. 
	
	The \emph{total curvature} of a polygon $P$ represented by edge vectors $e_1, \dots , e_n$ is the quantity
	\[
		\kappa(P) = \sum_{i=1}^n \phi(e_i, e_{i+1}),
	\]
	where $\phi(v,w)$ is the angle between the vectors $v$ and $w$ (in particular, $\phi(e_i,e_{i+1})$ is the $i$th exterior angle of $P$) and we compute indices cyclically so that $e_{n+1} = e_1$. To compare total curvatures for polygons with different lengths, we will use the \emph{average turning angle} for $P$, which is simply the average of the $\phi(e_i,e_{i+1})$, namely $\frac{\kappa(P)}{n}$.

	For unconfined equilateral polygons, Grosberg~\cite{grosbergTotalCurvatureTotal2008} showed that the expected total curvature is
	\[
		\mathbb{E}_{\Polhat(n)}[\kappa] = n \frac{\pi}{2} + \frac{3\pi}{8} + O\left(\frac{1}{n}\right);
	\]
	an integral formula for the exact value of $\mathbb{E}_{\Polhat(n)}[\kappa]$ is given in~\cite[Theorem~12]{cantarellaSymplecticGeometryClosed2016}. When edge vectors are chosen independently (that is, in a random walk), the expected turning angle is exactly $\frac{\pi}{2}$; intuitively, the excess in Grosberg's formula is due to the fact that each edge in a closed polygon feels some pressure to return to the starting point. Notice that the average turning angle is decreasing as a function of $n$, as one would expect: the closure constraint is felt by all edges equally, so with more edges the effect on each individual edge is smaller.

	In tight confinement consecutive edges can almost never point in even approximately the same direction, so we expect the typical turning angle to be larger, and hence for total curvature to increase; see, e.g.,~\cite{cantarellaSymplecticGeometryClosed2016,diaoTotalCurvatureTotal2018,diaoCurvatureRandomWalks2013}.

	Using \Cref{alg:sampling}, we generated 1,000,000 random elements of $\Polhat(n;1)$ for $n$ being multiples of 500 up to 20,000 and computed their total curvatures. The sample averages of turning angle are shown in \Cref{fig:total curvature} (left). We tried to fit this average turning angle data to a function of the form $a + \frac{b}{n}$ and found that we get an excellent fit ($R^2 > 1-10^{-11}$) to the function $f(n) = 2.14625-\frac{0.46742}{n}$. This is somewhat bigger than but broadly compatible with Diao, Ernst, Rawdon, and Ziegler's estimate of $\frac{\pi}{2} + 0.57 \approx 2.1408$ for the asymptotic average turning angle of polygons in confinement radius $1$~\cite[Figure~2(a)]{diaoTotalCurvatureTotal2018}.

	Having found this fit, we then generated 1,000,000 random elements of $\Polhat(n;1)$ for each $n=10,\dots , 500$ in steps of 10. The sample averages of turning angle are shown in \Cref{fig:total curvature} (right) along with the graph of $f(n)$. The fit is still very good, despite the fact that our derivation of $f(n)$ knew nothing about this data. The exact values of average turning angle for all runs are given in \Cref{sec:data}.

	\begin{figure}[t]
		\centering
			\includegraphics[width=3in]{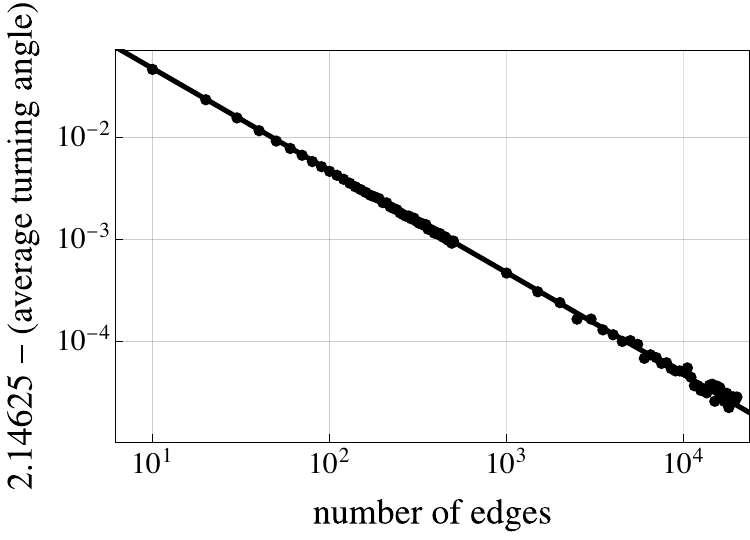}
		\caption{The log--log plot of $2.14625 - (\text{average turning angle})$ for all of our data. The line is the graph of $2.14625 - f(n) = \frac{0.46742}{n}$. We do not show the best-fit line, which has slope $\approx -0.99154$, because it aligns so well to the graph that it would be effectively invisible in the plot.}
		\label{fig:log fit}
	\end{figure}

	\Cref{fig:log fit} shows a log--log plot of $2.14625 - (\text{average turning angle})$ for all of our data, along with the graph of $2.14625 - f(n) = \frac{0.46742}{n}$ and the best fit line. The best fit line has slope $\approx -0.99154$, which is compatible with the hypothesis that $-1$ is the correct power of $n$. Overall, we see strong evidence from the data that the average turning angle of polygons in $\Polhat(n;1)$ fits the model $2.14625-\frac{0.46742}{n}$.

	As one more piece of evidence for this model, we now give an argument that the asymptotic limit should be $\approx 2.14625$.
	
	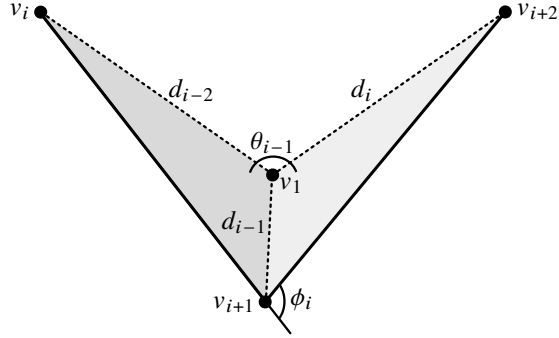
\begin{figure}[t]
		\centering
			\includegraphics[width=2.5in]{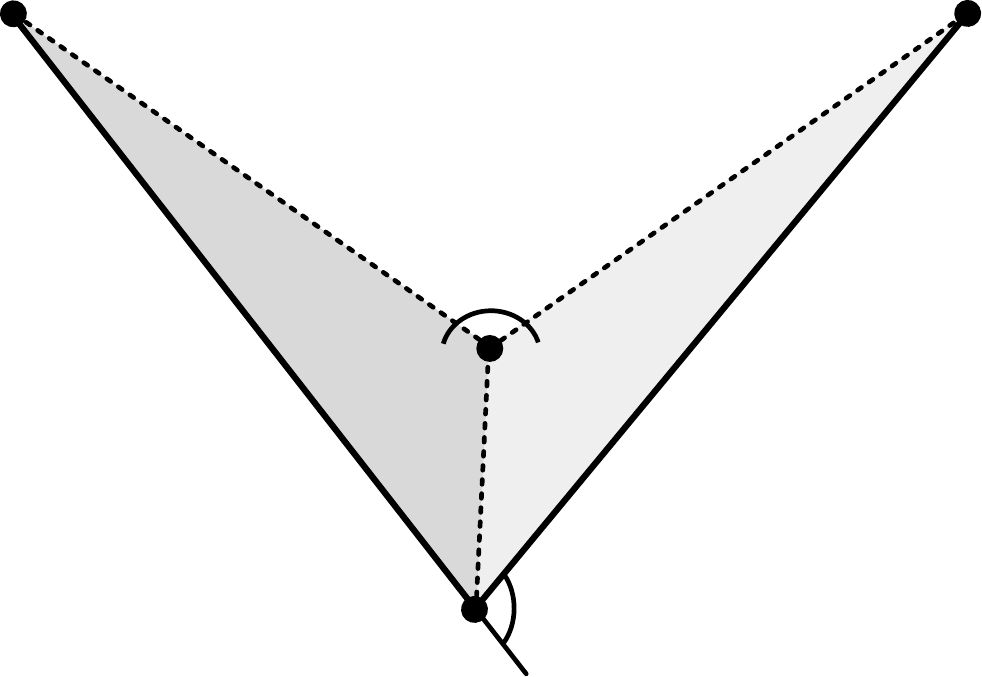}
			\put(-113,10){\smash{$v_{i+1}$}}
			\put(-190,121){\smash{$v_{i}$}}
			\put(2,121){\smash{$v_{i+2}$}}
			\put(-130,90){\smash{$d_{i-2}$}}
			\put(-110,40){\smash{$d_{i-1}$}}
			\put(-61,90){\smash{$d_{i}$}}
			\put(-84,11){\smash{$\phi_i$}}
			\put(-99,70){\smash{$\theta_{i-1}$}}
			\put(-88,55){\smash{$v_1$}}
		\caption{A portion of an equilateral $n$-gon, showing vertices $v_1, v_i, v_{i+1},$ and $v_{i+2}$, along with diagonal lengths $d_{i-2}, d_{i-1},$ and $d_{i}$, the dihedral angle $\theta_{i-1}$, and the $i$th turning angle $\phi_i$.}
		\label{fig:turning angle}
	\end{figure}

	Consider the $i$th turning angle $\phi_i = \phi(e_i,e_{i+1})$, and the neighboring edges, vertices, and diagonals, as illustrated in \Cref{fig:turning angle}. We imagine that this is a small piece of an equilateral $n$-gon with $n$ very large, so that we are in the asymptotic limit. Since the action-angle coordinates $((d_1, \dots , d_{n-3}),(\theta_1, \dots , \theta_{n-3}))$ give coordinates on $\Polhat(n;1)$, we can compute the asymptotic expected value of $\phi_i$ by expressing it as a function of the action-angle coordinates, and then integrating against the asymptotic distribution.

	\begin{proposition}\label{prop:phi}
		The turning angle $\phi_i$ can be expressed in terms of action-angle coordinates as
		\begin{equation}\label{eq:phi}
			\phi_i = \arccos\left(\frac{d_{i-2}^2+d_i^2}{2} - 1 - \frac{(d_{i-2}^2 + d_{i-1}^2 - 1)(d_{i-1}^2 + d_{i}^2 - 1) - A \cos \theta_{i-1}}{4 d_{i-1}^2}\right),
		\end{equation}
		where 
		\begin{multline*}
			A = \sqrt{(d_{i-2}+d_{i-1}-1)(d_{i-2}-d_{i-1}+1)(-d_{i-2}+d_{i-1}+1)(d_{i-2}+d_{i-1}+1)}  \\
			\times \sqrt{(d_{i-1}+d_{i}-1)(d_{i-1}-d_{i}+1)(-d_{i-1}+d_{i}+1)(d_{i-1}+d_{i}+1)}.
		\end{multline*}
	\end{proposition}
	
	We defer the proof of this proposition until the end of the section.

	The goal is to integrate the expression \eqref{eq:phi} against the distribution of $((d_1, \dots , d_{n-3}),(\theta_1, \dots , \theta_{n-3}))$. Since \eqref{eq:phi} depends only on $d_{i-2},d_{i-1},d_i,$ and $\theta_{i-1}$, we can integrate out the remaining variables and we should get a quadruple integral over the variables $(d_{i-2},d_{i-1},d_i,\theta_{i-1})$. 
	
	The random variate $\theta_{i-1}$ is independent of the rest and has a very simple distribution: it is uniform on $[0,2\pi)$.

	In the asymptotic limit, \Cref{cor:asymptotic chordlength} implies that the diagonal lengths $d_{i-2}, d_{i-1}$, and $d_i$ will be chosen from the limiting density $f(t) = 1- \cos(\pi t)$. However, if we just take independent draws from this distribution, we have no guarantee that they will satisfy the defining inequalities $1 \leq d_{i-2} + d_{i-1}$ and $1 \leq d_{i-1} + d_i$.
	
	To guarantee that the inequalities are satisfied, we can generate $d_{i-2}$ with density $f$, and then generate $d_{i-1}$ with conditional density $f_{d_{i-2}}$ and $d_i$ with conditional density $f_{d_{i-1}}$, proving the following.
	
	\begin{theorem}\label{thm:asymptotic turning angle}
		The expected turning angle in the asymptotic limit is
			\begin{multline}\label{eq:Ephi}
				\mathbb{E}_\infty[\phi_i] = \frac{1}{2\pi} \int_0^{2\pi}\!\!\!\!\int_0^1\!\!\! \int_0^1\!\!\!\int_0^1\!  \phi_i f(d_{i-2})f_{d_{i-2}}(d_{i-1}) f_{d_{i-1}}(d_i)   \operatorname{d}\!d_{i-2} \operatorname{d}\!d_{i-1} \operatorname{d}\!d_i\operatorname{d}\!\theta_{i-1} \\
				= \frac{\pi}{8} \int_0^{2\pi}\!\!\!\!\int_{1-d_{i-1}}^1\! \int_{1-d_{i-2}}^1\!\int_0^1\!  \phi_i (1-\cos(\pi d_{i-2}))  \frac{\sin\left(\frac{\pi}{2}d_{i}\right)}{\sin\left(\frac{\pi}{2}d_{i-2}\right)} \operatorname{d}\!d_{i-2} \operatorname{d}\!d_{i-1} \operatorname{d}\!d_i\operatorname{d}\!\theta_{i-1},
			\end{multline}
			where $\phi_i$ has the explicit expression \eqref{eq:phi}.
	\end{theorem}
	
	Numerically integrating this in \emph{Mathematica} yields $2.14625$ (to 5 decimal places), perfectly matching our fit for the asymptotic limit of average turning angle.

	All of this gives compelling evidence for the following conjecture on the form of the asymptotic expansion of expected turning angle for random polygons in $\Polhat(n;1)$.

	\begin{conjecture}\label{conj:turning angle}
		For large $n$, the expected $i$th turning angle of a random polygon in $\Polhat(n;1)$ is
		\[
			\mathbb{E}[\phi_i] = \mathbb{E}_\infty[\phi_i] - \frac{b}{n} + o\left(\frac{1}{n}\right),
		\]
		where $\mathbb{E}_\infty[\phi_i] \approx 2.14625$ is the explicit integral given in~\eqref{eq:Ephi}, and $b$ is a numerical constant $b \approx 0.46742$.
		
		Equivalently, the expected total curvature on $\Polhat(n;1)$ is
		\[
			\mathbb{E}[\kappa] = \mathbb{E}_\infty[\phi_i]n - b + o(1).
		\]
	\end{conjecture}

	We don't have a guess for an explicit formula for $b$.

	Unlike with unconfined polygons, the average turning angle increases with $n$ rather than decreasing. This phenomenon only seems to occur in very tight confinement (see, e.g.,~\cite[Figure~12]{cantarellaSymplecticGeometryClosed2016} and~\cite[Sections~5.1--5.2]{diaoTotalCurvatureTotal2018}).
	
	We conclude this section with the proof of \Cref{prop:phi}.
	
	\begin{proof}[{Proof of \Cref{prop:phi}}]
	Since the edges $e_i = v_{i+1}-v_i$ and $e_{i+1} = v_{i+2}-v_{i+1}$ have unit length, $\cos \phi_i$ is simply their dot product:
	\begin{equation}\label{eq:cosphi1}
		\cos\phi_i = (v_{i+1}-v_i) \cdot (v_{i+2}-v_{i+1}) = v_{i+1} \cdot v_{i+2} - v_i \cdot v_{i+2} - \|v_{i+1}\|^2 + v_i \cdot v_{i+1}.
	\end{equation}
	Since the root vertex $v_1$ is at the origin, 
	\begin{equation}\label{eq:norm vi+1}
		\|v_{i+1}\|^2 = d_{i-1}^2
	\end{equation} 
	and 
	\[
		1 = \|e_i\|^2 = (v_{i+1} - v_i) \cdot (v_{i+1} - v_i) = \|v_{i+1}\|^2 - 2 v_i \cdot v_{i+1} + \|v_i\|^2 = d_{i-1}^2 - 2 v_i \cdot v_{i+1} + d_{i-2}^2,
	\]
	so 
	\begin{equation}\label{eq:consecutive dots}
		v_i \cdot v_{i+1} = \frac{d_{i-1}^2 + d_{i-2}^2 - 1}{2} \quad \text{and similarly} \quad v_{i+1} \cdot v_{i+2} = \frac{d_{i}^2 + d_{i-1}^2 - 1}{2}.
	\end{equation}
	Plugging \eqref{eq:norm vi+1} and \eqref{eq:consecutive dots} into \eqref{eq:cosphi1} yields
	\begin{equation}\label{eq:cosphi}
		\cos \phi_i = \frac{d_{i-2}^2 + d_i^2}{2} - 1 - v_i \cdot v_{i+2}.
	\end{equation}
	
	On the other hand, the dihedral angle $\theta_{i-1}$ is the angle between the unit normal vectors $\frac{v_i \times v_{i+1}}{\|v_i \times v_{i+1}\|}$ and $\frac{v_{i+1} \times v_{i+2}}{\|v_{i+1} \times v_{i+2}\|}$ to the two triangles, so
	\begin{equation}\label{eq:costheta1}
		\cos\theta_{i-1} = \frac{(v_i \times v_{i+1})\cdot(v_{i+1} \times v_{i+2})}{\|v_i \times v_{i+1}\|\|v_{i+1} \times v_{i+2}\|}.
	\end{equation}
	Using the Binet–Cauchy identity, the numerator 
	\begin{multline}\label{eq:quadruple product}
		(v_i \times v_{i+1})\cdot(v_{i+1} \times v_{i+2}) = (v_i \cdot v_{i+1})(v_{i+1} \cdot v_{i+2}) - (v_i \cdot v_{i+2})(v_{i+1} \cdot v_{i+1}) \\
		= \frac{1}{4}(d_{i-2}^2 + d_{i-1}^2 - 1)(d_{i-1}^2 + d_{i}^2 - 1) - d_{i-1}^2 v_i \cdot v_{i+2}
	\end{multline}
	where we've used \eqref{eq:norm vi+1} and \eqref{eq:consecutive dots} to express three of the dot products in terms of the $d$'s.
	
	In the denominator of~\eqref{eq:costheta1}, $\|v_i \times v_{i+1}\|$ is twice the area of the triangle with side lengths $d_{i-2}, d_{i-1}$, and~$1$. We can compute this area using Heron's formula, yielding
	\[
		\|v_i \times v_{i+1}\| = \frac{1}{2} \sqrt{(d_{i-2}+d_{i-1}-1)(d_{i-2}-d_{i-1}+1)(-d_{i-2}+d_{i-1}+1)(d_{i-2}+d_{i-1}+1)}.
	\]
	Substituting this, the analogous formula for $\|v_{i+1}\times v_{i+2}\|$, and \eqref{eq:quadruple product} into \eqref{eq:costheta1} yields
	\begin{equation}\label{eq:costheta}
		 \cos\theta_{i-1}=\frac{(d_{i-2}^2 + d_{i-1}^2 - 1)(d_{i-1}^2 + d_{i}^2 - 1) - 4d_{i-1}^2 (v_i \cdot v_{i+2})}{A}.
	\end{equation}

	Solving~\eqref{eq:costheta} for $v_i \cdot v_{i+2}$ and plugging into~\eqref{eq:cosphi} gives an expression for $\cos \phi_i$:
	\[
		\cos\phi_i = \frac{d_{i-2}^2+d_i^2}{2} - 1 - \frac{(d_{i-2}^2 + d_{i-1}^2 - 1)(d_{i-1}^2 + d_{i}^2 - 1) - A \cos \theta_{i-1}}{4 d_{i-1}^2}.
	\]
	Taking the inverse cosine of both sides completes the proof.
	\end{proof}

\section{Questions and Future Directions} 
\label{sec:conclusion}

From our point of view, the most pressing question coming out of this work is whether it is possible to modify CPOP to cover confinement radii other than $R=1$. It is of substantial interest how confinement radius affects both geometric features like total curvature and topological features like knotting~\cite{diaoAverageCrossingNumber2018,diaoCurvatureRandomWalks2013,diaoKnotSpectrumConfined2014,diaoRelativeFrequenciesAlternating2018,diaoTotalCurvatureTotal2018,ernstKnottingSpectrumPolygonal2021}, and a generalized CPOP that could handle a range of $R$ would provide a new tool for exploring these problems.

For example, since $\mathcal{P}_n(R) \subseteq \mathcal{P}_n(R')$ when $R \leq R'$, similar combinatorial interpretations of $\mathcal{P}_n(R)$ for integer $R$ might enable reasonably efficient rejection samplers even for non-integer $R$. In an unpublished paper~\cite{chapmanOrthoschemesCoordinateSystem2019}, Chapman and Schreyer gave a combinatorial orthoscheme decomposition of $\mathcal{P}_n$, though it is somewhat unwieldy in practice. This might be a useful starting point for further exploration.

It would be satisfying to get a more elementary description of the quantity $E_\infty[\phi_i] \approx 2.14625$ than that in~\eqref{eq:Ephi} or to get any sort of analytic description of the quantity $b \approx 0.46742$ appearing in \Cref{conj:turning angle}. Indeed, without more understanding of these quantities, a proof of \Cref{conj:turning angle} seems out of reach. 

More generally, in the $R=1$ regime large-$n$ numerical explorations of other geometric quantities---like total torsion or writhe---or of topological features---like knotting---are now possible. In particular, since complicated knots are likely to be highly compact~\cite{diaoKnotSpectrumConfined2014,diaoRelativeFrequenciesAlternating2018,mooreTopologicallyDrivenSwelling2004,orlandiniAsymptoticsKnottedLattice1998}, polygons generated by CPOP are likely to give a form of ``enriched sampling'' for complicated knots (cf.~\cite{eddyNewStickNumber2022}). Moreover, it is generally believed that tightly confined polygons preferentially form prime knots~\cite{ernstKnottingSpectrumPolygonal2021}, as opposed to unconfined polygons, which are exponentially likely to be composite~\cite{diaoKnottingEquilateralPolygons1995}; as a model for random knots, where does $\Polhat(n;1)$ fit into Even-Zohar's taxonomy~\cite{even-zoharModelsRandomKnots2017}?

\subsection*{Acknowledgments}

Thanks to Kyle Chapman and Erik Schreyer, whose unpublished paper~\cite{chapmanOrthoschemesCoordinateSystem2019} provided inspiration, to the many other colleagues with whom we have discussed random polygons over the years, especially Jason Cantarella, and to the anonymous referees, whose detailed, thoughtful, and enthusiastic comments have made this a much better paper. We are very grateful for OEIS~\cite{oeis}, without which this paper probably would not exist. This work was supported by the National Science Foundation (DMS--2107700).

\clearpage

\appendix
\section{Average Turning Angle Data} 
\label{sec:data}

	Below we record the average turning angle for 1 million random elements of $\Polhat(n;1)$ for various $n$. The first block (which was actually the second run) covers $n$ from 10 to 500 in steps of 10; the second block covers $n$ from 500 to 20,000 in steps of 500.

	\begin{center}
	\begin{multicols*}{3}
	\TrickSupertabularIntoMulticols

	\tablefirsthead{
		\multicolumn{1}{c}{$n$} & \multicolumn{1}{l}{Avg. turning angle} \\
		\midrule
	}
	\tablehead{
		\multicolumn{1}{c}{$n$} & \multicolumn{1}{l}{Avg. turning angle} \\
		\midrule
	}
	\tablelasttail{\bottomrule}

	\begin{supertabular*}{.26\textwidth}{rl}
		10 & 2.100581721 \\
		20 & 2.123224279 \\
		30 & 2.130964208 \\
		40 & 2.134792121 \\
		50 & 2.137180449 \\
		60 & 2.138579043 \\
		70 & 2.139669385 \\
		80 & 2.140538045 \\
		90 & 2.141161817 \\
		100 & 2.141670407 \\
		110 & 2.142057442 \\
		120 & 2.142417816 \\
		130 & 2.142743199 \\
		140 & 2.143021670 \\
		150 & 2.143218218 \\
		160 & 2.143402482 \\
		170 & 2.143577180 \\
		180 & 2.143677481 \\
		190 & 2.143771587 \\
		200 & 2.143987619 \\
		210 & 2.144002625 \\
		220 & 2.144192762 \\
		230 & 2.144268263 \\
		240 & 2.144325655 \\
		250 & 2.144455228 \\
		260 & 2.144523104 \\
		270 & 2.144585018 \\
		280 & 2.144580270 \\
		290 & 2.144682053 \\
		300 & 2.144659530 \\
		310 & 2.144765877 \\
		320 & 2.144829440 \\
		330 & 2.144830929 \\
		340 & 2.144888697 \\
		350 & 2.144874572 \\
		360 & 2.145009015 \\
		370 & 2.145016327 \\
		380 & 2.145041241 \\
		390 & 2.145112802 \\
		400 & 2.145095060 \\
		410 & 2.145144604 \\
		420 & 2.145126882 \\
		430 & 2.145189363 \\
		440 & 2.145218436 \\
		450 & 2.145208467 \\
		460 & 2.145258700 \\
		470 & 2.145288265 \\
		480 & 2.145298440 \\
		490 & 2.145346165 \\
		500 & 2.145301198 \\
		\midrule
		500 & 2.145309867 \\
		1000 & 2.145790897 \\
		1500 & 2.145949853 \\
		2000 & 2.146015637 \\
		2500 & 2.146089473 \\
		3000 & 2.146089097 \\
		3500 & 2.146124566 \\
		4000 & 2.146137949 \\
		4500 & 2.146153888 \\
		5000 & 2.146152141 \\
		5500 & 2.146160267 \\
		6000 & 2.146185323 \\
		6500 & 2.146179808 \\
		7000 & 2.146184234 \\
		7500 & 2.146192641 \\
		8000 & 2.146191759 \\
		8500 & 2.146199162 \\
		9000 & 2.146202009 \\
		9500 & 2.146202016 \\
		10,000 & 2.146202867 \\
		10,500 & 2.146198483 \\
		11,000 & 2.146208793 \\
		11,500 & 2.146216447 \\
		12,000 & 2.146216357 \\
		12,500 & 2.146220053 \\
		13,000 & 2.146219403 \\
		13,500 & 2.146221864 \\
		14,000 & 2.146216008 \\
		14,500 & 2.146215223 \\
		15,000 & 2.146227213 \\
		15,500 & 2.146216835 \\
		16,000 & 2.146218053 \\
		16,500 & 2.146225334 \\
		17,000 & 2.146227263 \\
		17,500 & 2.146222148 \\
		18,000 & 2.146230611 \\
		18,500 & 2.146224827 \\
		19,000 & 2.146224484 \\
		19,500 & 2.146226968 \\
		20,000 & 2.146224623 \\
		\end{supertabular*}
	\end{multicols*}
	\end{center}

\printbibliography[heading=bibintoc]

\end{document}